\documentclass[final,leqno]{siamltex}

\usepackage{amsmath,amssymb,amscd,amsxtra,amsfonts,mathrsfs}
\usepackage{epsf,graphicx,epsfig,color,latexsym,cite,cases}

\newtheorem{theo}{Theorem}[section]
\newtheorem{coro}[theo]{Corollary}
\newtheorem{lemm}[theo]{Lemma}

\newtheorem{rema}[theo]{Remark}
\newtheorem{defi}[theo]{Definition}

\numberwithin{equation}{section}

\title{Stability  for  Time-domain Elastic Wave Equations}

\author{
Bochao Chen\thanks{School of Mathematics and Statistics, Center for
Mathematics and Interdisciplinary Sciences, Northeast Normal University,
Changchun, Jilin 130024, China. This author's research was supported in part by NSFC grants (project number, 11901232) and the Fundamental Research Funds for the Central Universities (project number, 2412022QD032). ({\tt
chenbc758@nenu.edu.cn})} \and
Yixian Gao\thanks{School of Mathematics and Statistics, Center for
Mathematics and Interdisciplinary Sciences, Northeast Normal University,
Changchun, Jilin 130024, China. This author's research was supported in part by NSFC grants (project number, 11871140) and National Key R\&D Program of China (project number, 2020YFA0714102). ({\tt
gaoyx643@nenu.edu.cn})}
\and Shuguan Ji\thanks{School of Mathematics and Statistics, Center for
Mathematics and Interdisciplinary Sciences, Northeast Normal University,
Changchun, Jilin 130024, China. This author's research was supported in part by NSFC grants (project numbers, 11671071 and 12071065). ({\tt jisg100@nenu.edu.cn})}
\and Yang Liu\thanks{School of Mathematics and Statistics, Center for
Mathematics and Interdisciplinary Sciences, Northeast Normal University,
Changchun, Jilin 130024, China. ({\tt liuy694@nenu.edu.cn})}}

\begin{document}

\maketitle

\begin{abstract}
This paper is concerned with  the inverse  scattering  problem involving the time-domain elastic wave equations
 in a bounded  $d$-dimensional domain.
First, an explicit reconstruction formula for the density  is established  by means of the  Dirichlet-to-Neumann operator.
The reconstruction is mainly based on the  modified boundary control method and complex geometric optics solutions for the elastic wave.
Next, the stable observability is  obtained by a Carleman estimate. Finally,  the stability for the density is presented by the connect operator.
\end{abstract}

\begin{keywords}
Elastic wave equations, Boundary control method, Complex geometric optics solutions, Carleman estimate
\end{keywords}

\begin{AMS}
35R30, 35B35
\end{AMS}

\pagestyle{myheadings}
\thispagestyle{plain}
\markboth{B. Chen, Y. Gao, S. Ji and Y. Liu}{Stability for the  Time-domain  Elastic Wave Equations}

\section{Introduction}

The inverse scattering problem of  bounded structures for acoustic, electromagnetic, and elastic waves has aroused the interest of physicists, engineers, and applied mathematicians, and occurs in significant applications in various scientific areas, for example, in optics, acoustics, radio wave propagation, radar techniques, geophysical prospecting, and medical imaging. We refer to  \cite{colton2019} for details of these applications.

Acoustic and elastic wave equations are two fundamental equations to describe the wave propagation. Most of the studies in the literature are devoted to the inverse scattering problem for acoustic wave equations, see \cite{Beretta2016,Belish1987, Li2021,lizhao2020,Beilina2018,Gao2018}. Different from the acoustic wave, the elastic wave, such as the seismic wave, is composed of the shear wave and the compressional wave. In practice, we have to face the unpredictability of the environments and the lack of knowledge of elastic waves. Hence the inverse scattering problem for elastic wave equations are more complicated. This makes our problem more meaningful and challenging.
In recent years, the research on elastic waves has become more and more extensive.
The global uniqueness and stability results for the inverse medium problem of time-harmonic elastic wave equations have  been established by  \cite{barcelo2018} and \cite{Beretta2017}.
The relevant results of homogeneous isotropic elastic wave equations can be referred to \cite{Bardos1992,Lagnese1983,Beli2002}. In \cite{Ikehata1998}, the authors studied the inverse problem of determining the force term or the density from a finite number of measurements of lateral boundary data. In \cite{Chen2003}, the authors obtained the Lipschitz stability for a principally diagonal hyperbolic system by a Carleman estimate. Recently, the uniqueness result associated with nonlinear isotropic elastic equations was given by \cite{de2020,Uhlmann2021}.
To our best knowledge, there are no stability results in the existing literature when the density of the time-domain elastic system is reconstructed.

In this paper, we focus on a time-dependent elastic wave equation with variable coefficients in a bounded domain, and establishing the stability result for its density  from the associated active measurement. For this, we have to consider the uniqueness and stability of the corresponding inverse scattering problem. There are many methods to prove the uniqueness and stability of inverse scattering problems, such as the boundary control (BC) method introduced by \cite{Belish1987}, complex geometric optics (CGO) solutions originated from \cite{G1987} and so on. Generally, by using either of the two methods mentioned above, we can only obtain a logarithmic stability estimate, see \cite{tatru1995,Mandache2001}. A further generalization on a Lipschitz stability estimate for the wave speed has been given by \cite{liu2016} thanks to modified BC method together with incorporates features from the CGO solutions approach. Motivated by \cite{liu2016}, we state a reconstruction formula for the density and give the corresponding Lipschitz stability  via known boundary measurements modelled by  Dirichlet-to-Neumann operators.
The construction formula depends on the Blagove\v{s}\v{c}enski\u{\i} type identities and the CGO solutions corresponding to the elastic case, and then can be applied to obtain a Lipschitz estimate for low frequencies. On the other hand,
because of the exponential constant in the Lipschitz inequality, we can get a  logarithmic stability for high frequencies. In following work, our arguments depend on the assumption on the stable observability for the elastic wave equation since the inverse medium problem is nonlinear.
Meanwhile, stable observability relies on a observability inequality of the elastic case, which is equivalent to the exact controllability of the elastic case. In \cite{Ala1999, Shang2018}, the authors just proved the exact controllability of anisotropic for the homogeneous elastodynamic system and the inhomogeneous case.
In our case, we need to establish a Carleman estimate to state the observability inequality and the stable observability for elastic case, and then give stability results for density.
 In addition, we have to refer to \cite{Chen2003,Ikehata1998,Imanuvilov2020}. The authors established a Carleman estimate for the  Lam\'{e} system with variable coefficients which is regarded as a principally diagonal hyperbolic system.
However, such a form of the Carleman estimate for a diagonal system cannot be combined with the boundary control operator we derived.
In order to give a more general form of the Carleman estimate for
a hyperbolic system with variable coefficients, we have to  require further additional but natural conditions on Lam\'{e} parameters.

This paper is organized as follows. Section \ref{2} aims at  presenting a precise mathematical formulation of the model  scattering problem for the elastic wave and introducing some notations which will be used throughout the paper.
The goal of Section \ref{3} is to reconstruct a formula associated with the density.  In Section \ref{4},
we give a Carleman estimate for the elastic system, and then prove the corresponding stable observability.
In Section \ref{5}, we establish the local  Lipschitz stability  for low frequencies   and the logarithmic stability estimate for  high frequencies, respectively.

\section{Problem formulation}\label{2}
In this section, we introduce a mathematical model and give some notations for the scattering problem in a bounded domain.
\subsection{Model problem}
Let $\Omega\subset \mathbb{R}^d, 3\leq d<\infty$ be a compact set  with smooth enough boundary $\partial\Omega$.
Consider the initial-boundary value problem of the time-domain elastic system
\begin{align}\label{ab}
\begin{cases}
 \partial_t^2\boldsymbol u(t, \boldsymbol x)-\mathcal L_\rho\boldsymbol u(t, \boldsymbol x)=\boldsymbol 0 \quad &\text{in}~(0, T)\times\Omega,\\
  \boldsymbol u(0, \boldsymbol x)=\boldsymbol 0,\quad \partial_t\boldsymbol u(0, \boldsymbol x)= \boldsymbol 0 \quad &\text{in}~\Omega ,\\
  \boldsymbol u(t, \boldsymbol x)=\boldsymbol f(t, \boldsymbol x) \quad &\text{on}~(0, T)\times\partial\Omega
 \end{cases}
\end{align}
with
\begin{align*}
\mathcal L_\rho\boldsymbol u:= \frac{1}{\rho}\left(\nabla\cdot(\mu(\nabla\boldsymbol u+\nabla\boldsymbol u^\top))+\nabla(\lambda\nabla\cdot\boldsymbol u)\right).
\end{align*}
Here
  $\rho:=\rho(\boldsymbol x)$   is the density with $\rho\in \mathcal C^\infty(\Omega)$, and $\mu:=\mu(\boldsymbol x), \lambda:=\lambda(\boldsymbol x)$ are Lam\'{e} parameters with $\mu,\lambda\in \mathcal C^1(\Omega)$, satisfying
\begin{align*}
d\lambda+2\mu>0.
\end{align*}
Moreover, there exist constants $\rho_1$, $\mu_i$ and $\lambda_i, i=0, 1$  such that
\begin{align*}
0<\rho\leq\rho_1, \quad 0<\mu_0\leq\mu\leq\mu_1,\quad \lambda_0\leq\lambda\leq \lambda_1.
\end{align*}

Denote by $\boldsymbol{u}(t,\boldsymbol x)=(u_1(t,\boldsymbol x), u_2(t,\boldsymbol x),\cdots, u_d(t,\boldsymbol x))^{\top}$ the displacement vector with  $t\in (0, T)$ and $\boldsymbol x=(x_1, x_2, \cdots,x_d)^{\top} \in  \Omega$.  The  gradient tensor $\nabla \boldsymbol u$ and the divergence $\nabla\cdot\boldsymbol u$ are, respectively, defined as
\begin{align*}
\nabla \boldsymbol u=
\begin{bmatrix}
\partial_{x_1}u_1& \partial_{x_2}u_1& \dots & \partial_{x_d}u_1\\
\partial_{x_1}u_2& \partial_{x_2}u_2& \dots & \partial_{x_d}u_2\\
\vdots& \vdots& \ddots&\vdots\\
\partial_{x_1}u_d& \partial_{x_2}u_d&\dots &  \partial_{x_d}u_d
\end{bmatrix},~\quad \nabla\cdot\boldsymbol u=\underset{i=1}{\overset{d}{\sum}}\partial_{x_i}u_i.
\end{align*}
Furthermore, $\nabla\boldsymbol u^\top$  denotes the transpose of  $\nabla\boldsymbol u$.

Denote by $\boldsymbol u_{f}(t, \boldsymbol x)=\boldsymbol u(t, \boldsymbol x)$ the solution of \eqref{ab} corresponding to the boundary condition $\boldsymbol f$ with  the compatibility condition $\boldsymbol f(0, \boldsymbol x)=\boldsymbol 0$. Let $\boldsymbol {\mathcal{C}}^{\infty} (\partial \Omega):={\mathcal{C}}^{\infty} (\partial \Omega)^d$ signify the Cartesian product space equipped with
the corresponding norm. Similarly, $\boldsymbol  H^1 (\partial \Omega):= H^1 (\partial \Omega)^d$ and $\boldsymbol L^2 (\partial \Omega):=L^2(\partial \Omega)^d.$
For $\boldsymbol f\in \boldsymbol {\mathcal{C}}^\infty_0((0, T)\times\partial\Omega)$,
where
\begin{align*}
 \boldsymbol {\mathcal{C}}^\infty_0((0, T)\times\partial\Omega):=\{\boldsymbol f\in \boldsymbol {\mathcal{C}}^{\infty}((0, T)\times \partial\Omega):\boldsymbol f(0, \boldsymbol x)=\boldsymbol 0\},
\end{align*}
define formally the displacement to traction map as
\begin{align*}
 \Lambda_{\rho, T}[\boldsymbol f]=T_{\boldsymbol \nu}\boldsymbol u_f,
\end{align*}
where $\boldsymbol \nu$ is the outward unit
normal vector at $\boldsymbol x\in \partial\Omega$,
and $T_{\boldsymbol \nu} \boldsymbol u_f$ is the surface traction given by
\begin{align*}
	T_{\boldsymbol \nu} \boldsymbol u_f:=\mu\partial_{\boldsymbol \nu} \boldsymbol u_f+\mu(\nabla\boldsymbol u_f)^\top\cdot\boldsymbol \nu+\lambda(\nabla\cdot \boldsymbol u_f)\boldsymbol \nu|_{((0, T)\times\partial\Omega)}.
\end{align*}
Indeed, if  taking $\boldsymbol f \in \boldsymbol  H_0^1((0, T)\times \partial\Omega)$,
where
\begin{align*}
\boldsymbol H_0^1((0, T)\times \partial\Omega):=\{\boldsymbol f\in \boldsymbol H^1((0, T)\times \partial\Omega):\boldsymbol f(0, \boldsymbol x)=\boldsymbol 0\},
\end{align*}
then it follows from  \cite{Lasi1986} that the
traction map
\[\Lambda_{\rho, T}: \boldsymbol H_0^1((0, T)\times \partial\Omega)\rightarrow  \boldsymbol L^2((0, T)\times \partial\Omega)\]
is continuous.

When $\boldsymbol f$ is regarded as a boundary source, the operator $\Lambda_{\rho,T}$ models boundary measurements for the elastic wave produced by the source on  $(0, T)\times\partial\Omega$.
The inverse problem for the elastic wave equation is described as follows:

{\bf(IP)} Reconstructing the density $\rho$  by giving the knowledge of $\Lambda_{\rho, T}$.

From the computational point of view, a more challenging issue is the lack of stability.
A small variation of the data may bring a huge error in the reconstruction.
Thus, the stability issue for the above {\bf(IP)} problem has gained extensive attention.
However, since the propagation mode of elastic wave includes the superposition of shear wave and compressional wave, the corresponding scattering problem becomes more complicated.


\subsection{Notations}
The Fourier transform of the function $\boldsymbol u(\boldsymbol x)$ is defined by
\begin{align*}
\mathcal F(\boldsymbol u)(\boldsymbol\xi)=\int_{\mathbb R^d} e^{{\rm i} \boldsymbol x\cdot\boldsymbol \xi}\boldsymbol u(\boldsymbol x) {\rm d}\boldsymbol x, ~\quad \boldsymbol x\in\mathbb R^d.
\end{align*}
Denote by $\boldsymbol L^2(\Omega): = L^2(\Omega)^d$ and $\boldsymbol H_0^1(\Omega): =H_0^1 (\Omega)^d  $ the Sobolev spaces equipped with the corresponding  norms
\begin{align*}
\|\boldsymbol u\|_{\boldsymbol L^2(\Omega)}&=\big(\int_\Omega |\boldsymbol u|^2{\rm d} \boldsymbol x\big)^{1/2}, &&\boldsymbol x\in \Omega,\\
\|\boldsymbol u\|^2_{\boldsymbol H_0^1(\Omega)}&=\|\nabla \boldsymbol u\|^2_{\boldsymbol L^2(\Omega)},&&\boldsymbol x\in \Omega.
\end{align*}
Define the inner products
\begin{align*}
(\boldsymbol u, \boldsymbol v)_{ \boldsymbol  L^2((0, T)\times \Omega)}&=\int_0^T\int_\Omega\boldsymbol u\cdot\boldsymbol v {\rm d}\boldsymbol x{\rm d}t,&& (t, \boldsymbol x)\in (0, T)\times\Omega,\\
\langle\boldsymbol f, \boldsymbol h\rangle_{ \boldsymbol L^2((0, T)\times\partial\Omega)}&=\int_0^T\int_{\partial\Omega}\boldsymbol f\cdot\boldsymbol h{\rm d}S{\rm d} t,&& (t, \boldsymbol x)\in (0, T)\times\partial\Omega,
\end{align*}
and
\begin{align*}
	(\boldsymbol u, \boldsymbol v)_{\boldsymbol L^2(\Omega;\rho{\rm d}\boldsymbol x)}&=\int_\Omega(\boldsymbol u\cdot\boldsymbol v) \rho {\rm d}\boldsymbol x,&& \boldsymbol x\in \Omega,\\
(\boldsymbol u, \boldsymbol v)_{\boldsymbol L^2(\Omega)}&=\int_\Omega\boldsymbol u\cdot\boldsymbol v {\rm d}\boldsymbol x,&& \boldsymbol x\in \Omega,\\
\langle\boldsymbol f, \boldsymbol h\rangle_{\boldsymbol L^2(\partial\Omega)}&=\int_{\partial\Omega}\boldsymbol f\cdot\boldsymbol h {\rm d}S,&& \boldsymbol x\in \partial\Omega.
\end{align*}

Let the operator $\Theta$ be the extension of time by zero from $(0, T)$ to $(0, 2T)$ given by
\begin{align}\label{theta}
\hat{\boldsymbol f}(t, \cdot): =\Theta \boldsymbol f(t, \cdot), ~\quad \boldsymbol f\in \boldsymbol L^2((0, T)\times\partial\Omega), ~t\in (0, T).
\end{align}
The integral  operators $\mathscr B$ and  $\mathscr I$ are defined by
\begin{align}
	\mathscr B\boldsymbol f(t, \cdot):&=\frac{1}{2}\int^{2T-t}_t\boldsymbol f(s, \cdot){\rm d}s,&& \boldsymbol f\in  \boldsymbol L^2((0, 2T)\times\partial\Omega), ~t\in (0, T),\label{B}\\
	\mathscr I\boldsymbol f(t, \cdot):&=\int^T_t\boldsymbol f(s, \cdot){\rm d}s,&&\boldsymbol f\in \boldsymbol L^2((0, T)\times\partial\Omega),~t\in (0, T).\nonumber
\end{align}


\section{Reconstruct a formula for the density}\label{3}

The goal of this section is to apply a modified boundary control method and the CGO solutions approach to obtain the  reconstruction formula associated with the  density.

\subsection{An identity for the wave and the source}

This subsection is devoted to establishing the Blagove\v{s}\v{c}enski\u{\i} type identity of the vector case, which are the foundation of the boundary control method introduced by \cite{Bla1966}, while the  scalar form is given in \cite{Bingham2008}.

Let $\boldsymbol u_{f}$, $\boldsymbol u_{h}$ stand for the solution of \eqref{ab} with respect to boundary $\boldsymbol f$ and $\boldsymbol h$, respectively. The following lemma corresponds to the Blagove\v{s}\v{c}enski\u{\i} type identity.
\begin{lemm}\label{blag}
Suppose that $\boldsymbol f, \boldsymbol h\in  \boldsymbol {\mathcal C}^\infty_0((0, T)\times\partial\Omega)$. One has
\begin{align}
 (\boldsymbol u_{f}(T),\boldsymbol u_{h}(T))_{\boldsymbol L^2(\Omega;\rho{\rm d}\boldsymbol x)}=\langle\boldsymbol f,\mathcal J(\Lambda_{\rho,2T})\boldsymbol h\rangle_{\boldsymbol L^2((0, T)\times\partial\Omega)},
\label{ag}
\end{align}
 where $\mathcal J(\Lambda_{\rho, 2T}):=\Lambda_{\rho, T}^{\ast} \mathscr B\Theta- \mathscr B\Lambda_{\rho, 2T}\Theta$ and  $\Lambda_{\rho, T}^{\ast}$ is the adjoint operator of $\Lambda_{\rho, T}$.
\end{lemm}
\begin{proof}
Extending $\boldsymbol f$ and $\boldsymbol h$ from $(0, T)\times\partial\Omega$ to $(0, 2T)\times \partial\Omega$, denote by $\hat{\boldsymbol f},  \hat{\boldsymbol h}$.
Let
\begin{align*}
 	w(t, s) = \int_\Omega(\boldsymbol  u_{\hat f} (t, \cdot)\cdot\boldsymbol u_{\hat h}(s, \cdot))\rho(\cdot){\rm d}\boldsymbol x.
\end{align*}
One has
\begin{align}\label{differeq}
&(\partial^2_t-\partial^2_s)w(t, s)\nonumber\\
=&\int_\Omega\big(\nabla\cdot(\mu(\nabla\boldsymbol u_{\hat f}(t, \cdot)+\nabla\boldsymbol u_{\hat f}^\top(t, \cdot)))+\nabla(\lambda\nabla\cdot \boldsymbol u_{\hat f}(t, \cdot))\big)\cdot\boldsymbol u_{\hat h} (s, \cdot){\rm d}\boldsymbol x\nonumber\\
&-\int_\Omega\boldsymbol u_{\hat f}(t,\cdot)\cdot\big(\nabla\cdot(\mu(\nabla\boldsymbol u_{\hat h}(s, \cdot)+\nabla\boldsymbol u_{\hat h}^\top(s, \cdot)))+\nabla(\lambda\nabla\cdot \boldsymbol u_{\hat h}(s, \cdot))\big){\rm d}\boldsymbol x  \nonumber\\
 =&\int_{\partial\Omega}\big(\mu\partial_{\boldsymbol \nu}\boldsymbol u_{\hat f}(t, \cdot)+\mu\nabla\boldsymbol u_{\hat f}^\top(t, \cdot)\cdot\boldsymbol \nu
 +\lambda(\nabla\cdot\boldsymbol u_{\hat f} (t, \cdot))\boldsymbol \nu\big)\cdot\hat{\boldsymbol h}(s, \cdot){\rm d}S\nonumber\\
  &-\int_\Omega\mu(\nabla \boldsymbol u_{\hat f}(t, \cdot)+\nabla \boldsymbol u_{\hat f}^\top(t, \cdot)):\nabla \boldsymbol u_{\hat h}(s, \cdot)+\lambda(\nabla \cdot\boldsymbol u_{\hat f}(t, \cdot))(\nabla \cdot\boldsymbol u_{\hat h}(s, \cdot)) {\rm d}\boldsymbol x\nonumber\\
&-\int_{\partial\Omega}\hat{\boldsymbol f}(t, \cdot)\cdot\big(\mu\partial_{\boldsymbol \nu}\boldsymbol u_{\hat h}(s, \cdot)+\mu\nabla\boldsymbol u_{\hat h}^\top(s, \cdot)\cdot\boldsymbol \nu+\lambda(\nabla\cdot\boldsymbol u_{\hat h}(s, \cdot))\boldsymbol \nu\big){\rm d}S\nonumber\\
 &+\int_\Omega\nabla \boldsymbol u_{\hat f}(t, \cdot):\mu(\nabla \boldsymbol u_{\hat h}(s, \cdot)+\nabla \boldsymbol u_{\hat h}^\top(s, \cdot))+(\nabla \cdot\boldsymbol u_{\hat f}(t, \cdot))(\lambda\nabla \cdot\boldsymbol u_{\hat h}(s, \cdot)){\rm d}\boldsymbol x\nonumber\\
 =&\langle\Lambda_{\rho, 2T}[\hat{\boldsymbol f}(t, \cdot)], \hat{\boldsymbol h}(s, \cdot)\rangle_{\boldsymbol L^2(\partial\Omega)}-\langle\hat{\boldsymbol f}(t, \cdot),\Lambda_{\rho, 2T}[\hat{\boldsymbol h}(s, \cdot)]\rangle_{\boldsymbol L^2(\partial\Omega)}.
\end{align}
Here $\boldsymbol A: \boldsymbol B =\text{tr}(\boldsymbol A\boldsymbol B^\top)$ is the Frobenius inner product of square matrices $\boldsymbol A$ and $\boldsymbol B$.

We regard \eqref{differeq} as a inhomogeneous  partial  differential equation
\begin{align*}
	\partial^2_tw(t, s)-\partial^2_sw(t, s)=\langle\Lambda_{\rho, 2T}[\hat{\boldsymbol f}(t, \cdot)],  \hat{\boldsymbol h}(s, \cdot)\rangle_{\boldsymbol L^2(\partial\Omega)}-\langle\hat{\boldsymbol f}(t, \cdot), \Lambda_{\rho, 2T}[\hat{\boldsymbol h}(s, \cdot)]\rangle_{\boldsymbol L^2(\partial\Omega)}
\end{align*}
with  vanishing initial conditions $w(t, s)|_{t=0}=\partial_t w(t, s)|_{t=0}=0.$
 It follows from \cite[Theorem 3.4.1]{Struss1992} that
\begin{align*}
w(t, s)=&\frac{1}{2}(w(t, s+t)-w(t, s-t))|_{t=0}+\frac{1}{2}\int^{s+t}_{s-t}\partial_{t'} w(t', s')|_{t'=0}{\rm d}s'\\
&+\frac{1}{2}\int_{0}^t\int_{s-(t-t')}^{s+(t-t')}\langle\Lambda_{\rho, 2T}[\hat{\boldsymbol f}(t', \cdot)],  \hat{\boldsymbol h}(s', \cdot)\rangle_{\boldsymbol L^2(\partial\Omega)}\\
&-\langle\hat{\boldsymbol f}(t', \cdot), \Lambda_{\rho, 2T}[\hat{\boldsymbol h}(s', \cdot)]\rangle_{\boldsymbol L^2(\partial\Omega)} {\rm d}s'{\rm d}t'\\
=&\frac{1}{2}\int_{0}^t\int_{s-(t-t')}^{s+(t-t')}\langle\Lambda_{\rho, 2T}[\hat{\boldsymbol f}(t', \cdot)],  \hat{\boldsymbol h}(s', \cdot)\rangle_{\boldsymbol L^2(\partial\Omega)}\\
&-\langle\hat{\boldsymbol f}(t', \cdot), \Lambda_{\rho, 2T}[\hat{\boldsymbol h}(s', \cdot)]\rangle_{\boldsymbol L^2(\partial\Omega)} {\rm d}s'{\rm d}t'.
\end{align*}
If we take  $t=s=T$ and following  the definitions of operator $\Theta$ and $\mathscr B$ in \eqref{theta}--\eqref{B}, then
\begin{align*}
&w(T,T)\\
=&\frac{1}{2}\int^{T}_0\int^{2T-t}_{t}\langle\Lambda_{\rho, 2T}[\hat{\boldsymbol f}(t, \cdot)],  \hat{\boldsymbol h}(s, \cdot)\rangle_{\boldsymbol L^2(\partial\Omega)}-\langle\hat{\boldsymbol f}(t, \cdot), \Lambda_{\rho, 2T}[\hat{\boldsymbol h}(s, \cdot)]\rangle_{\boldsymbol L^2(\partial\Omega)}
{\rm d}s{\rm d}t\\
=&\frac{1}{2}\int^{T}_0\int^{2T-t}_{t}\langle\Lambda_{\rho, 2T}\Theta{\boldsymbol f}(t, \cdot),  \Theta{\boldsymbol h}(s, \cdot)\rangle_{\boldsymbol L^2(\partial\Omega)}-\langle\Theta{\boldsymbol f}(t, \cdot), \Lambda_{\rho, 2T}\Theta{\boldsymbol h}(s, \cdot)\rangle_{\boldsymbol L^2(\partial\Omega)}
{\rm d}s{\rm d}t\\
=&\int^{T}_0\langle\Lambda_{\rho, 2T}\Theta{\boldsymbol f}(t, \cdot),  \mathscr B\Theta{\boldsymbol h}(t, \cdot)\rangle_{\boldsymbol L^2(\partial\Omega)}-\langle\boldsymbol f(t, \cdot), \mathscr B\Lambda_{\rho, 2T}\Theta{\boldsymbol h}(t, \cdot)\rangle_{\boldsymbol L^2(\partial\Omega)}{\rm d}t\\
=&\int^{T}_0\langle\Lambda_{\rho, T}[{\boldsymbol f}(t, \cdot)],  \mathscr B\Theta{\boldsymbol h}(t, \cdot)\rangle_{\boldsymbol L^2(\partial\Omega)}-\langle\boldsymbol f(t, \cdot), \mathscr B\Lambda_{\rho, 2T}\Theta{\boldsymbol h}(t, \cdot)\rangle_{\boldsymbol L^2(\partial\Omega)}{\rm d}t\\
=&\int^{T}_0\langle\boldsymbol f(t, \cdot),  \Lambda_{\rho, T}^*\mathscr B\Theta{\boldsymbol h}(t, \cdot)\rangle_{\boldsymbol L^2(\partial\Omega)}-\langle\boldsymbol f(t, \cdot), \mathscr B\Lambda_{\rho, 2T}\Theta{\boldsymbol h}(t, \cdot)\rangle_{\boldsymbol L^2(\partial\Omega)}{\rm d}t\\
=&\int^{T}_0\langle\boldsymbol f(t, \cdot),  \Lambda_{\rho, T}^*\mathscr B\Theta{\boldsymbol h}(t, \cdot)-\mathscr B\Lambda_{\rho, 2T}\Theta{\boldsymbol h}(t, \cdot)\rangle_{\boldsymbol L^2(\partial\Omega)}{\rm d}t.
\end{align*}
Therefore,
\begin{align*}
(\boldsymbol u_{f}(T),&\boldsymbol u_h(T))_{\boldsymbol L^2(\Omega; \rho{\rm d}\boldsymbol x)}
	=\int^{T}_0\langle\boldsymbol f(t, \cdot),  \Lambda_{\rho, T}^*\mathscr B\Theta{\boldsymbol h}(t, \cdot)-\mathscr B\Lambda_{\rho, 2T}\Theta{\boldsymbol h}(t, \cdot)\rangle_{\boldsymbol L^2(\partial\Omega)}{\rm d}t.
\end{align*}
The proof is complete.
\end{proof}

Then we give a modified Blagove\v{s}\v{c}enski\u{\i} type identity.
\begin{theo}
Suppose that $\boldsymbol f\in \boldsymbol {\mathcal{C}}^\infty_0((0, T)\times\partial\Omega)$, and $\boldsymbol \phi\in \boldsymbol { \mathcal{C}}^\infty(\Omega)$ with
$\boldsymbol \phi(\boldsymbol x):=\boldsymbol\iota e^{\rm i\boldsymbol\theta\cdot \boldsymbol x}$, where $\boldsymbol { \mathcal{C}}^\infty(\Omega):={\mathcal{C}}^{\infty} (\Omega)^d,\boldsymbol \iota\in \mathbb C^d$ and
\begin{align*}
\boldsymbol \theta= (\boldsymbol\xi+\rm i\boldsymbol\eta)/2,~\quad  \boldsymbol \xi,~\boldsymbol\eta\in \mathbb{R}^d, \quad|\boldsymbol \xi|=|\boldsymbol \eta|.
\end{align*}
Furthermore, we assume that $\boldsymbol\iota$ and $\boldsymbol\theta$ satisfy
\begin{align}\label{cgo}
	{\rm i}(\boldsymbol\theta \cdot\boldsymbol \iota^\top+\boldsymbol \iota\cdot\boldsymbol \theta^\top)\nabla \mu+{\rm i}(\boldsymbol \iota\cdot\boldsymbol\theta)\nabla \lambda-\mu(\boldsymbol \theta\cdot\boldsymbol\theta\boldsymbol)\boldsymbol\iota-(\lambda+\mu)(\boldsymbol \iota\cdot\boldsymbol\theta)\boldsymbol\theta=\boldsymbol 0
\end{align}
with $|\boldsymbol \iota|\neq 0$.
 Then
\begin{align}
 (\boldsymbol u_{f}(T), \boldsymbol \phi)_{\boldsymbol L^2(\Omega;  \rho{\rm d}\boldsymbol{x})}=\langle\boldsymbol f, \mathcal K(\Lambda_{\rho, T})\boldsymbol \phi\rangle_{\boldsymbol L^2((0, T)\times\partial\Omega)},
\label{ah}
\end{align}
where
\begin{align}\label{E:K}
\mathcal K(\Lambda_{\rho, T}):= \Lambda_{\rho, T}^*\mathscr IT_0- \mathscr IT_1,
\end{align}
and $T_i,i=0, 1$ are the first two traces on $\partial\Omega$
defined by
\begin{align*}
 T_0\boldsymbol \phi=\boldsymbol \phi|_{\partial\Omega},
 ~\quad T_1\boldsymbol \phi=(\mu\partial_{\boldsymbol \nu}\boldsymbol \phi+\mu\nabla\boldsymbol \phi^\top\cdot\boldsymbol \nu+\lambda(\nabla\cdot\boldsymbol \phi)\boldsymbol \nu)|_{\partial\Omega}.
\end{align*}
\end{theo}
\begin{proof}
By proceeding an analogous analysis as Lemma \ref{blag}, we immediately see that
\begin{align*}
 \partial^2_t(\boldsymbol u_{f}, \boldsymbol \phi)_{\boldsymbol L^2(\Omega; \rho{\rm d}\boldsymbol x)}=&(\nabla\cdot(\mu(\nabla\boldsymbol u_f+\nabla\boldsymbol u_f^\top))+\nabla(\lambda\nabla\cdot \boldsymbol u_{f}),  \boldsymbol \phi)_{\boldsymbol L^2(\Omega)}\\
 &-(\boldsymbol u_f, \nabla\cdot(\mu(\nabla \boldsymbol \phi+\nabla\boldsymbol \phi^\top)) +\nabla(\lambda\nabla\cdot \boldsymbol \phi))_{\boldsymbol L^2(\Omega)}\\
 =&\langle\mu\partial_{\boldsymbol \nu}\boldsymbol u_f+\mu\nabla\boldsymbol u_f^\top\cdot\boldsymbol \nu+\lambda(\nabla\cdot\boldsymbol u_f)\boldsymbol \nu,  \boldsymbol \phi \rangle_{\boldsymbol L^2(\partial\Omega)}\\
 &-\langle\boldsymbol f, \mu\partial_{\boldsymbol \nu}\boldsymbol \phi+\mu\nabla \boldsymbol \phi^\top\cdot\boldsymbol \nu+\lambda(\nabla\cdot\boldsymbol \phi)\boldsymbol \nu\rangle_{\boldsymbol L^2(\partial\Omega)}\\
 =&\langle\Lambda_{\rho, T}[\boldsymbol f],  \boldsymbol \phi\rangle_{\boldsymbol L^2(\partial\Omega)}-\langle\boldsymbol f,  \mu\partial_{\boldsymbol \nu}\boldsymbol \phi+\mu\nabla \boldsymbol \phi^\top\cdot\boldsymbol \nu+\lambda(\nabla\cdot\boldsymbol \phi)\boldsymbol \nu\rangle_{\boldsymbol L^2(\partial\Omega)}.
\end{align*}
Using  initial conditions $\boldsymbol u(t, \cdot)|_{t=0}=\partial_t \boldsymbol u(t, \cdot)|_{t=0}=\boldsymbol  0$,
we can derive
\begin{align*}
 &(\boldsymbol u_{f}(T),
 \boldsymbol \phi)_{\boldsymbol  L^2(\Omega;  \rho{\rm d}\boldsymbol x)}\\
=&\int^T_0\int^s_0\langle\Lambda_{\rho, T}[\boldsymbol f],  \boldsymbol \phi\rangle_{\boldsymbol L^2(\partial\Omega)}-\langle\boldsymbol f,  \mu\partial_{\boldsymbol \nu}\boldsymbol \phi+\mu\nabla \boldsymbol \phi^\top\cdot\boldsymbol \nu+\lambda(\nabla\cdot\boldsymbol \phi)\boldsymbol \nu\rangle_{\boldsymbol L^2(\partial\Omega)} {\rm d}t  {\rm d}s\\
=&\langle\mathcal{I}\Lambda_{\rho, T}[\boldsymbol f],  \boldsymbol \phi\rangle_{\boldsymbol L^2((0, T)\times\partial\Omega)}-\langle\mathcal{I}\boldsymbol f,  \mu\partial_{\boldsymbol \nu}\boldsymbol \phi+\mu\nabla \boldsymbol \phi^\top\cdot\boldsymbol \nu+\lambda(\nabla\cdot\boldsymbol \phi)\boldsymbol \nu\rangle_{\boldsymbol L^2((0, T)\times\partial\Omega)}\\
=&\langle\boldsymbol f,  \Lambda^*_{\rho, T} \mathcal{I}^*\boldsymbol \phi\rangle_{\boldsymbol L^2((0, T)\times\partial\Omega)}-\langle\boldsymbol f, \mathcal{I}^*(\mu\partial_{\boldsymbol \nu}\boldsymbol \phi+\mu\nabla \boldsymbol \phi^\top\cdot\boldsymbol \nu+\lambda(\nabla\cdot\boldsymbol \phi)\boldsymbol \nu)\rangle_{\boldsymbol L^2((0, T)\times\partial\Omega)}\\
=&\langle\boldsymbol f, \Lambda^*_{\rho, T}\mathcal{I}^*\boldsymbol \phi-\mathcal{I}^*(\mu\partial_{\boldsymbol \nu}\boldsymbol \phi+\mu\nabla \boldsymbol \phi^\top\cdot\boldsymbol \nu+\lambda(\nabla\cdot\boldsymbol \phi)\boldsymbol \nu)\rangle_{\boldsymbol L^2((0, T)\times\partial\Omega)},
\end{align*}
where $\mathcal{I}\boldsymbol f(s, \boldsymbol x):=\int^s_0\boldsymbol f(t, \boldsymbol x) {\rm d}t$.
It is easy to verify that $\mathcal{I}^*=\mathscr I$.

The proof is complete.
\end{proof}
\begin{rema}
The function $\boldsymbol \phi$ is a CGO  solution for the  elastostatic system
    \begin{align*}
    	\nabla\cdot(\mu(\nabla \boldsymbol \phi+\nabla\boldsymbol \phi^\top)) +\nabla(\lambda\nabla\cdot \boldsymbol \phi)= \boldsymbol 0.
    \end{align*}
For instance, let $\boldsymbol\xi$ and $\boldsymbol \eta$ satisfy $\boldsymbol \xi\bot\boldsymbol \eta.$
If $\boldsymbol \iota\cdot\boldsymbol \theta=0$ and $\mu$ is a constant, then the condition \eqref{cgo} holds.
\end{rema}


\subsection{Properties of the connect operator}
Denote by $\mathscr D:  \boldsymbol L^2((0, T)\times \partial\Omega)\rightarrow  \boldsymbol L^2((0, T)\times \Omega)$ the Dirichlet map corresponding to \eqref{ab} as follows:
\begin{align} \label{Dirich}
\mathscr D \boldsymbol g=\boldsymbol {z},
\end{align}
where $ \boldsymbol {z}$ satisfies the boundary value problem
\begin{align*}
\begin{cases}
\mathcal L_\rho\boldsymbol {z}=\boldsymbol 0~\quad &\text{in}~(0, T)\times\Omega,\\
\boldsymbol {z}=\boldsymbol {g}~\quad &\text{on}~(0, T)\times\partial\Omega.
\end{cases}
\end{align*}
This is a transient deterministic solution problem, although formally dependent on time, which is solved by regarding time as a parameter.
It is well known from the standard elliptic theory that
$\mathscr D$ is a continuous operator from $\boldsymbol L^2((0, T)\times\partial\Omega)$ to $ \boldsymbol L^2((0, T)\times\Omega)$.

The following Lemmas \ref{bouncotro}--\ref{anW} address some facts to be used in the sequel, proofs of which can be found in \cite{Beli2002}.

\begin{lemm}\label{bouncotro}
The solution $\boldsymbol u(t, \boldsymbol x)$ to the problem \eqref{ab} is written as the following
formula:
\begin{align*}
	\boldsymbol u(t, \boldsymbol x)=-\mathcal L_\rho\int_0^t S(t-\tau)\mathscr D\boldsymbol f(\tau, \boldsymbol x) {\rm d}\tau~\quad \text{for} ~\boldsymbol f\in \boldsymbol { \mathcal C}_{0}^\infty((0, T)\times\partial\Omega),
\end{align*}
where $\mathscr D$ is the Dirichlet operator as seen in \eqref{Dirich} and  $S(t)$ is a family of sines generated by the operator $\mathcal L_\rho$.
\end{lemm}

\begin{lemm}\label{anW}
The operator $\mathcal W: \boldsymbol L^2((0, T)\times\partial\Omega)\rightarrow  \boldsymbol L^2(\Omega)$, given by
\begin{align*}
\mathcal W\boldsymbol f=-\int_0^T\mathcal L_\rho S(T-t)	\mathscr D\boldsymbol f(t, \boldsymbol x) {\rm d}t \quad \text{for}~ \boldsymbol f\in \boldsymbol L^2((0, T)\times\partial\Omega),
\end{align*}
is continuous.

Moreover,  its adjoint operator $\mathcal W^*: \boldsymbol L^2(\Omega)\rightarrow  \boldsymbol L^2((0, T)\times\partial\Omega)$, defined by
\begin{align*}
\mathcal W^*\boldsymbol{\mathfrak u}=-\mathscr D^*\mathcal L_\rho^*S^*(T-t)	\boldsymbol{\mathfrak u}~\quad \text{for}~\boldsymbol{\mathfrak u}\in \boldsymbol L^2(\Omega),
\end{align*}
is also continuous.
\end{lemm}

Let us define a control map $\mathcal W_{\rho, T}: \boldsymbol L^2((0, T)\times \partial\Omega)\rightarrow  \boldsymbol L^2(\Omega)$ as follows:
\begin{align}\label{ctm}
\mathcal W_{\rho, T}\boldsymbol f:=\boldsymbol u_{f}(T),
\end{align}
where  $\boldsymbol u_f$ is the solution of \eqref{ab}.
Thanks to Lemmas \ref{bouncotro}--\ref{anW}, it is easy to see that $\mathcal  W_{\rho, T}$ is continuous.

Besides, substituting \eqref{ctm} into \eqref{ag} yields that
\begin{align*}
\mathcal J(\Lambda_{\rho, 2T})=\mathcal W_{\rho, T}^{*}\mathcal W_{\rho, T},
\end{align*}
where the connect operator
\begin{align*}
\mathcal J(\Lambda_{\rho, 2T}): \boldsymbol L^2((0, T)\times \partial\Omega)\rightarrow \boldsymbol L^2((0, T)\times \partial\Omega)
\end{align*}
extends as a continuous operator on $ \boldsymbol L^2((0, T)\times\partial\Omega)$ because of the $\boldsymbol L^2$ continuity of $\mathcal W_{\rho, T}$.
Additionally,  formula \eqref{ah} shows that $\mathcal K(\Lambda_{\rho, T})$ as seen in \eqref{E:K} is equivalent to  the restriction of $\mathcal W^*_{\rho, T}$ on the CGO solution of the elastostatic system.

Let us describe the exact controllability of the elastic wave equation by the following observability inequality.
\begin{defi}\label{def:inequality}
Let $\boldsymbol u$ be a solution of
\begin{align}\label{dual}
\begin{cases}
\partial_t^2\boldsymbol u(t, \boldsymbol x)-\mathcal L_\rho\boldsymbol u(t, \boldsymbol x)= \boldsymbol 0 \quad &\text{in}~(0, T)\times\Omega,\\
\boldsymbol u(t, \boldsymbol x)= \boldsymbol 0 \quad &\text{on}~(0, T)\times\partial\Omega.
\end{cases}
\end{align}
If there exists a constant $C_{obs}>0$  such that for $\Gamma\subset\partial\Omega$ and $T>0$,
\begin{align}
\|\partial_t \boldsymbol u(T)\|_{\boldsymbol L^2(\Omega)}+\|\boldsymbol u(T)\|_{\boldsymbol H_0^1(\Omega)}\leq C_{obs}\|\mu\partial_{\boldsymbol \nu}\boldsymbol u+\mu\nabla\boldsymbol u^\top\cdot\boldsymbol \nu+\lambda(\nabla\cdot\boldsymbol u)\boldsymbol \nu\|_{\boldsymbol L^2((0, T)\times\Gamma)},\label{obs}
\end{align}
then the elastic wave equation in \eqref{dual} is
 continuously observable.
\end{defi}
\begin{rema}
The precise definition of $\Gamma$ in Definition \ref{def:inequality} is as seen in \eqref{gamma}.
\end{rema}

Setting $\boldsymbol \varphi\in \boldsymbol  L^2(\Omega)$,  we  consider the following control equation
\begin{align}
 \mathcal W_{\rho, T} \boldsymbol f=\boldsymbol \varphi \quad \text{for}~\boldsymbol f\in  \boldsymbol L^2((0, T)\times\partial\Omega).\label{ai}
\end{align}
It is obvious that $\mathcal W_{\rho, T}$ is not injective. Then  we aim at solving $\boldsymbol f=\mathcal W_{\rho, T}^\dag\boldsymbol \varphi$  to replace \eqref{ai}, where $\mathcal W_{\rho, T}^\dag$ is the pseudo-inverse operator of $\mathcal W_{\rho, T}$.
If $\mathcal W_{\rho, T}$ is surjective, then the operator $\mathcal W_{\rho, T}^\dag$ is a bounded operator.

%

Next we introduce the following fact.
\begin{lemm}\label{le:surjective}
If the continuous observability holds with $\Gamma=\partial\Omega$, the map
\begin{align*}
	\mathcal N:\boldsymbol L^2((0, T)\times\partial\Omega)&\rightarrow \boldsymbol H^{-1}(\Omega)\times \boldsymbol L^2(\Omega),\\
	\boldsymbol f&\mapsto(\boldsymbol u_{f}(T), \partial_t\boldsymbol u_{f}(T))
\end{align*}
is surjective, where $\boldsymbol  H^{-1} ( \Omega):= H^{-1} ( \Omega)^d$.
\end{lemm}
\begin{proof}
Let  $\boldsymbol u$ be  a solution of
\begin{align*}
\begin{cases}
\partial_t^2\boldsymbol u(t, \boldsymbol x)-\mathcal L_\rho\boldsymbol u(t, \boldsymbol x)=\boldsymbol 0 \quad &\text{in}~ (0, T)\times\Omega,\\
\boldsymbol u(T, \boldsymbol x)=\boldsymbol u(T), \quad \partial_t\boldsymbol u (T, \boldsymbol x)=\partial_t\boldsymbol u (T) \quad &\text{in}~  \Omega,\\
\boldsymbol u(t, \boldsymbol x)=\boldsymbol 0  \quad &\text{on}~  (0, T)\times\partial\Omega,
\end{cases}
\end{align*}
and define
\begin{align*}
(\mathcal N \boldsymbol f, (\boldsymbol u(T), \partial_t\boldsymbol u(T)))_{\boldsymbol L^2(\Omega;  \rho {\rm d}\boldsymbol x)}	:=\int_\Omega \big(\boldsymbol v(T)\partial_t\boldsymbol u(T)+\partial_t\boldsymbol v(T)\boldsymbol u(T)\big)\rho {\rm d}\boldsymbol x,
\end{align*}
where $\boldsymbol v$ is a solution of \eqref{ab}. Notice that
\begin{align*}
	0=&\int_\Omega\int_0^T \boldsymbol v\cdot(\partial_t^2\boldsymbol u-\mathcal L_\rho \boldsymbol u)\rho-\boldsymbol u\cdot(\partial_t^2\boldsymbol v-\mathcal L_\rho \boldsymbol v)\rho {\rm d}t {\rm d}\boldsymbol x\\
	=&\int_\Omega(\boldsymbol v\cdot\partial_t\boldsymbol u -\boldsymbol u\cdot\partial_t\boldsymbol v)\rho{\rm d}\boldsymbol x\big|_{0}^T-\int_0^T \int_{\partial\Omega}\boldsymbol v\cdot\big(\mu(\nabla\boldsymbol u+\nabla\boldsymbol u^\top)\cdot\boldsymbol \nu+\lambda(\nabla\cdot\boldsymbol u)\boldsymbol \nu\big) {\rm d}S
 {\rm d}t\\
	 &+\int_0^T \int_{\partial\Omega}\boldsymbol u\big(\mu(\nabla\boldsymbol v+\nabla\boldsymbol v^\top)\cdot\boldsymbol \nu+\lambda(\nabla\cdot\boldsymbol v)\boldsymbol \nu\big) {\rm d}S
 {\rm d}t\\
=&\int_\Omega \big(\boldsymbol v(T)\partial_t\boldsymbol u(T)+\partial_t\boldsymbol v(T)\boldsymbol u(T)\big)\rho {\rm d}\boldsymbol x\\
&-\int_0^T \int_{\partial\Omega}\boldsymbol f\cdot\big(\mu(\nabla\boldsymbol u+\nabla\boldsymbol u^\top)\cdot\boldsymbol \nu+\lambda(\nabla\cdot\boldsymbol u)\boldsymbol \nu\big) {\rm d}S
 {\rm d}t,
\end{align*}
which then gives
\begin{align*}
(\mathcal N \boldsymbol f, (\boldsymbol u(T), \partial_t\boldsymbol u(T)))_{\boldsymbol L^2(\Omega;  \rho {\rm d}\boldsymbol x)}=\langle\boldsymbol f, (\mu(\nabla\boldsymbol u+\nabla\boldsymbol u^\top)\cdot\boldsymbol \nu+\lambda\nabla\cdot\boldsymbol u\boldsymbol \nu)\rangle_{\boldsymbol L^2((0, T)\times\partial\Omega)}.
\end{align*}
Then  its adjoint operator $\mathcal N^*: \boldsymbol H^1_0(\Omega)\times \boldsymbol L^2(\Omega)\rightarrow  \boldsymbol L^2((0, T)\times \partial\Omega)$ satisfies
\begin{align*}
\mathcal N^*(\boldsymbol u(T), \partial_t\boldsymbol u(T))=\mu(\nabla\boldsymbol u+\nabla\boldsymbol u^\top)\cdot\boldsymbol \nu+\lambda\nabla\cdot\boldsymbol u\boldsymbol \nu.
\end{align*}
It follows from \eqref{obs} that
\begin{align*}
	\|\partial_t \boldsymbol u(T)\|_{\boldsymbol L^2(\Omega)}+\|\boldsymbol u(T)&\|_{\boldsymbol H_0^1(\Omega)}\leq C_{obs}\|\mu\partial_{\boldsymbol \nu}\boldsymbol u+\mu\nabla\boldsymbol u^\top\cdot\boldsymbol \nu+\lambda(\nabla\cdot\boldsymbol u)\boldsymbol \nu\|_{\boldsymbol L^2((0, T)\times\Gamma)}.
\end{align*}
Hence the operator $\mathcal N$ is surjective by Theorem 4.1 in \cite{Bardos1992}.

This ends the proof of the lemma.
\end{proof}

In view of Lemma \ref{le:surjective}, $\mathcal{W}_{\rho, T}$ is surjective.
Therefore, if the continuous observability inequality \eqref{obs} is valid, then
\[\mathcal  W^\dag_{\rho, T}: \boldsymbol L^2(\Omega)\rightarrow \boldsymbol L^2((0, T)\times\partial\Omega)\]
is continuous and $\mathcal W_{\rho, T}\mathcal W_{\rho, T}^\dag$ is the identity operator. Moreover, since the range $R(\mathcal W_{\rho, T})$ is closed, the range
\begin{align*}
	R(\mathcal J(\Lambda_{\rho, 2T}))=R(\mathcal W_{\rho, T}^*\mathcal W_{\rho, T})=R(\mathcal W_{\rho, T}^*)
\end{align*}
 of $\mathcal J(\Lambda_{\rho, 2T})$ is  closed as well. As a result, the pseudo-inverse operator $\mathcal J^\dag(\Lambda_{\rho, 2T})$ is continuous on $\boldsymbol L^2((0, T)\times \partial\Omega)$ and $\mathcal J(\Lambda_{\rho, 2T}) \mathcal J^\dag(\Lambda_{\rho, 2T})$ is the orthogonal projection onto $R(\mathcal J(\Lambda_{\rho, 2T}))$.
\subsection{A reconstruction formula for the  density}
This subsection aim at establishing a reconstruction formula for the  density.
\begin{theo}\label{t1}
Let $\rho\in \mathcal C^\infty(\Omega)$ be positive and assume that equation \eqref{ab} is continuously observable from $\partial\Omega$ in time $T>0$.
Define the following functions
\begin{align*}
 \boldsymbol \phi:=\boldsymbol\iota e^{\rm i\boldsymbol\theta\cdot \boldsymbol x},~\quad
 \boldsymbol \psi:=\overline{\boldsymbol\iota } e^{\rm i\overline{\boldsymbol \theta}\cdot \boldsymbol x}
\end{align*}
with $\boldsymbol \iota\in \mathbb C^d$ and
\begin{align*}
\boldsymbol \theta= (\boldsymbol\xi+\rm i\boldsymbol\eta)/2,~\quad  \boldsymbol \xi,~\boldsymbol\eta\in \mathbb{R}^d, \quad|\boldsymbol \xi|=|\boldsymbol \eta|.
\end{align*}
Suppose, furthermore, that $\boldsymbol\iota$ and $\boldsymbol\theta$ satisfy
\begin{align*}
	{\rm i}(\boldsymbol\theta \cdot\boldsymbol \iota^\top+\boldsymbol \iota\cdot\boldsymbol \theta^\top)\nabla \mu+{\rm i}(\boldsymbol \iota\cdot\boldsymbol\theta)\nabla \lambda-\mu(\boldsymbol \theta\cdot\boldsymbol\theta\boldsymbol)\boldsymbol\iota-(\lambda+\mu)(\boldsymbol \iota\cdot\boldsymbol\theta)\boldsymbol\theta=\boldsymbol 0
\end{align*}
with $|\boldsymbol \iota|\neq 0$.
Then
\begin{align}
\mathcal{F}(\rho)(\boldsymbol\xi)=\frac{1}{|\boldsymbol\iota|^2}\langle \mathcal J^\dag(\Lambda_{\rho, 2T}) \mathcal K(\Lambda_{\rho, T})\boldsymbol \phi,  \mathcal K(\Lambda_{\rho, T})\boldsymbol \psi\rangle_{\boldsymbol L^2((0, T)\times \partial\Omega)},\label{af}
\end{align}
where $\mathcal F(\rho)$ stands for the Fourier transform of the extension by zero of $\rho$ on $\mathbb R^d$.
\end{theo}
\begin{proof}
Taking the extension by zero of $\rho$ onto $\mathbb R^d$, still remark as $  \rho$.
It follows from the definitions of $\boldsymbol \phi$ and $\boldsymbol \psi$ that
\begin{align*}
 (\boldsymbol \phi, \boldsymbol \psi)_{\boldsymbol L^2(\Omega; \rho{\rm d}\boldsymbol x)}
 &=|\boldsymbol\iota|^2\int_\Omega e^{\rm i\boldsymbol x\cdot \boldsymbol\xi}\rho{\rm d}\boldsymbol x=|\boldsymbol\iota|^2\int_{\mathbb R^d} e^{\rm i\boldsymbol x\cdot \boldsymbol\xi} \rho {\rm d}\boldsymbol x=|\boldsymbol\iota|^2\mathcal{F}(\rho)(\boldsymbol\xi).
\end{align*}

Moreover, since
$\mathcal K(\Lambda_{\rho, T})$ is the restrictions of $\mathcal W_{\rho, T}^*$ on $\boldsymbol \phi, \boldsymbol \psi$ and 
\begin{align*}
\mathcal W_{\rho, T}^\dag=(\mathcal W_{\rho, T}^*\mathcal W_{\rho, T})^\dag \mathcal W_{\rho, T}^*,
\end{align*}
a direct calculation yields that
\begin{align*}
(\boldsymbol \phi,  \boldsymbol \psi)_{\boldsymbol L^2(\Omega; \rho {\rm d}\boldsymbol x)}=&(\mathcal W_{\rho, T} \mathcal W^\dag _{\rho, T}\boldsymbol \phi,  \boldsymbol \psi)_{\boldsymbol L^2(\Omega; \rho{\rm d}\boldsymbol x)}\\
=&\langle \mathcal W^\dag_{\rho, T} \boldsymbol \phi,  \mathcal K(\Lambda_{\rho, T})\boldsymbol \psi\rangle_{\boldsymbol L^2((0, T)\times\partial\Omega)}\\
=&\langle \mathcal J^\dag(\Lambda_{\rho, 2T}) \mathcal K(\Lambda_{\rho, T})\boldsymbol \phi,  \mathcal K(\Lambda_{\rho, T})\boldsymbol \psi\rangle_{\boldsymbol L^2((0, T)\times\partial\Omega)}.
\end{align*}
We have completed the proof of the theorem.
\end{proof}

\section{Stable observability for the elastic wave equation}\label{4}

In this section, we will present a  Carleman estimate  and  illustrate the stable observability of the elastic wave system.
In fact, in order to
 give the local Lipschitz stability, the system is required to be  stable observability due to the nonlinear property of the inverse medium problem.

To make the Carleman estimate valid,  we have to define an admissible set associated with Lam\'{e} coefficients as follows
\begin{align}\label{nessary}
	\mathscr H:=\Big\{\mu, \lambda\in\mathcal C^1(\Omega):&(\mu(\nabla\boldsymbol u+\nabla\boldsymbol u^\top)+\lambda(\nabla\cdot\boldsymbol u)I_d):\nabla \boldsymbol u \geq c_0|\nabla\boldsymbol u|^2, \quad \text{and}~\nonumber\\
	&(\mu(\nabla\boldsymbol u+\nabla\boldsymbol u^\top)+\lambda(\nabla\cdot\boldsymbol u)I_d):\nabla \boldsymbol u-\underset{i=1}{\overset{d}{\sum}}((\partial_{x_i}\mu(\nabla\boldsymbol u+\nabla\boldsymbol u^\top) \nonumber\\
	&+\partial_{x_i}\lambda (\nabla\cdot \boldsymbol u)I_d):\nabla \boldsymbol u)\partial_{x_i}l\geq c_1|\nabla\boldsymbol u|^2\Big\},
\end{align}
where $c_0,  c_1$ are two positive constants.

\subsection{A Carleman estimate for the elastic wave equation}

We first give the following  auxiliary result.

\begin{lemm}\label{le:car1}
For $\boldsymbol w\in \boldsymbol{\mathcal{C}}^2(\mathbb{R}\times\Omega)$, $l\in \mathcal{C}^2(\Omega)$, we set
$\boldsymbol v=e^{\tau l}\boldsymbol w$,
where $\tau>0$ is a constant. One has
\begin{align}\label{a}
S_1\cdot S_2\leq\frac{e^{2\tau l}}{2}\big|\rho\partial_{t}^2\boldsymbol w-\nabla\cdot(\mu(\nabla\boldsymbol w+\nabla\boldsymbol w^\top))-\nabla(\lambda\nabla\cdot\boldsymbol w)\big|^2
\end{align}
with
\begin{align*}
 	S_1:=&\rho\partial^2_t\boldsymbol v-\nabla\cdot(\mu (\nabla\boldsymbol v+\nabla \boldsymbol v^{\top}))- \nabla(\lambda \nabla\cdot\boldsymbol v)+\mu\tau\Delta l\boldsymbol v-\mu\tau^2|\nabla l|^2\boldsymbol v\nonumber\\
 	&+(\lambda+\mu)\tau\nabla\nabla l\cdot\boldsymbol v-(\lambda+\mu)\tau^2\nabla l\cdot \nabla l^\top\cdot\boldsymbol v\\
 	&+\tau(\nabla l\cdot \boldsymbol v^\top+\boldsymbol v\cdot\nabla l^\top)\cdot\nabla \mu+\tau\nabla l \cdot\boldsymbol v\nabla \lambda+2(\mu-1)\tau\nabla\boldsymbol v\cdot\nabla l\\
 	&+(\lambda+\mu)\tau\nabla\boldsymbol v^\top\cdot\nabla l+(\lambda+\mu)\tau\nabla\cdot\boldsymbol v\nabla l-\tau \mathscr C\boldsymbol v,\\
 	S_2:=&\tau(2\nabla\boldsymbol v\cdot\nabla l+ \mathscr C\boldsymbol v),
 \end{align*}
where $\mathscr C$ is a suitable positive constant.
\end{lemm}
\begin{proof}
From the identities
\begin{align*}
&e^{\tau l}\nabla\cdot(\mu(\nabla\boldsymbol w+\nabla\boldsymbol w^\top))\\
=&\nabla\cdot(\mu(\nabla \boldsymbol v+\nabla\boldsymbol v^\top))-\mu\tau\Delta l\boldsymbol v-2\mu\tau\nabla\boldsymbol v\cdot\nabla l+\mu\tau^2|\nabla l|^2\boldsymbol v\\
&-\mu\tau\nabla\nabla l\cdot \boldsymbol v-\mu\tau\nabla\cdot \boldsymbol v\nabla l-\mu\tau\nabla \boldsymbol v^\top\cdot\nabla l+\mu\tau^2\nabla l\cdot\nabla l^\top\cdot \boldsymbol v\\
&-\tau(\nabla l\cdot \boldsymbol v^\top+\boldsymbol v\cdot\nabla l^\top)\cdot\nabla \mu,
\end{align*}
and
\begin{align*}
e^{\tau l}\nabla(\lambda\nabla\cdot\boldsymbol w)=&\nabla(\lambda\nabla\cdot\boldsymbol v)-\lambda\tau\nabla\nabla l\cdot\boldsymbol v-\lambda\tau\nabla\cdot\boldsymbol v\nabla l-\lambda\tau\nabla\boldsymbol v^\top\cdot\nabla l\\
&+\lambda\tau^2\nabla l\cdot \nabla l^\top\cdot\boldsymbol v-\tau\nabla l \cdot\boldsymbol v\nabla \lambda,
\end{align*}
it shows that
\begin{align*}
&\frac{e^{2\tau l}}{2}\big|\rho\partial_{t}^2\boldsymbol w-\nabla\cdot(\mu(\nabla\boldsymbol w+\nabla\boldsymbol w^\top))-\nabla(\lambda\nabla\cdot\boldsymbol w)\big|^2\\
=&\frac{1}{2}\big|\rho\partial^2_t\boldsymbol v- \nabla\cdot(\mu(\nabla\boldsymbol v+\nabla \boldsymbol v^{\top}))-\nabla(\lambda \nabla\cdot\boldsymbol v)+\mu\tau\Delta l\boldsymbol v\\
&-\mu\tau^2|\nabla l|^2\boldsymbol v+(\lambda+\mu)\tau\nabla\nabla l\cdot\boldsymbol v-(\lambda+\mu)\tau^2\nabla l\cdot \nabla l^\top\cdot\boldsymbol v\\
&+\tau(\nabla l\cdot \boldsymbol v^\top+\boldsymbol v\cdot\nabla l^\top)\cdot\nabla \mu+\tau\nabla l \cdot\boldsymbol v\nabla \lambda-\tau\mathscr C\boldsymbol v\\
&+2\mu\tau\nabla\boldsymbol v\cdot\nabla l+(\lambda+\mu)\tau\nabla\boldsymbol v^\top\cdot\nabla l+(\lambda+\mu)\tau\nabla\cdot\boldsymbol v\nabla l+\tau\mathscr C\boldsymbol v\big|^2\\
=&\frac{1}{2}\big|\big(\rho\partial^2_t\boldsymbol v- \nabla\cdot(\mu(\nabla\boldsymbol v+\nabla \boldsymbol v^{\top}))-\nabla(\lambda \nabla\cdot\boldsymbol v)+\mu\tau\Delta l\boldsymbol v\\
&-\mu\tau^2|\nabla l|^2\boldsymbol v+(\lambda+\mu)\tau\nabla\nabla l\cdot\boldsymbol v-(\lambda+\mu)\tau^2\nabla l\cdot \nabla l^\top\cdot\boldsymbol v\\
&+\tau(\nabla l\cdot \boldsymbol v^\top+\boldsymbol v\cdot\nabla l^\top)\cdot\nabla \mu+\tau\nabla l \cdot\boldsymbol v\nabla \lambda\\
&+2(\mu-1)\tau\nabla\boldsymbol v\cdot\nabla l+(\lambda+\mu)\tau\nabla\boldsymbol v^\top\cdot\nabla l\\
&+(\lambda+\mu)\tau\nabla\cdot\boldsymbol v\nabla l-\tau\mathscr C\boldsymbol v\big)+\tau(2\nabla\boldsymbol v\cdot\nabla l+\mathscr C\boldsymbol v)\big|^2\\
=&\frac{1}{2}(|S_1|^2+|S_2|^2)+S_1\cdot S_2.
\end{align*}
The proof is complete.
\end{proof}

Let $\boldsymbol u$ be a solution of system \eqref{dual}. For $\mu, \lambda \in \mathscr H$,
 we define an energy function
\begin{align}\label{energy1}
E(t):=\int_\Omega|\sqrt{\rho}\partial_t\boldsymbol u|^2
+\big(\mu(\nabla \boldsymbol u+\nabla\boldsymbol u^\top)+\lambda(\nabla\cdot\boldsymbol u)I_d\big): \nabla\boldsymbol u {\rm d}\boldsymbol x, \quad  t\in [0, \infty),
\end{align}
where $I_d$ is a $d\times d$ unit matrix.
It follows from
\begin{align*}
E(t)-E(0)=&\int_0^t\partial_s E(s) {\rm d}s\\
=&\int_0^t\int_\Omega\frac{\partial}{\partial s}\big(|\sqrt{\rho}\partial_s\boldsymbol u|^2+(\mu(\nabla \boldsymbol u+\nabla\boldsymbol u^\top)+\lambda(\nabla\cdot\boldsymbol u)I_d): \nabla\boldsymbol u\big) {\rm d}\boldsymbol x {\rm d}s\\
=&\int_0^t\int_\Omega2\rho\partial^2_s\boldsymbol u\cdot\partial_s\boldsymbol u+2\mu(\nabla \boldsymbol u+\nabla\boldsymbol u^\top):\partial_s\nabla\boldsymbol u\\
&+2\lambda(\partial_s\nabla\cdot\boldsymbol u)(\nabla\cdot\boldsymbol u) {\rm d}\boldsymbol x {\rm d}s\\
=&2\int_0^t\int_\Omega(\nabla\cdot(\mu(\nabla\boldsymbol u+\nabla\boldsymbol u^\top))+\nabla(\lambda\nabla\cdot\boldsymbol u))\cdot\partial_s\boldsymbol u\\
&+\mu (\partial_s\nabla\boldsymbol u+\partial_s\nabla\boldsymbol u^\top):\nabla\boldsymbol u+\lambda(\partial_s\nabla\cdot\boldsymbol u)(\nabla\cdot\boldsymbol u) {\rm d}\boldsymbol x {\rm d}s\\
=&\int_0^t\int_\Omega-2\mu(\nabla\boldsymbol u+\nabla\boldsymbol u^\top):\partial_s\nabla\boldsymbol u-2\lambda(\nabla\cdot\boldsymbol u)(\partial_s\nabla\cdot\boldsymbol u)\\
&+2\mu(\nabla \boldsymbol u+\nabla\boldsymbol u^\top):\partial_s\nabla\boldsymbol u+2\lambda(\partial_s\nabla\cdot\boldsymbol u)(\nabla\cdot\boldsymbol u) {\rm d}\boldsymbol x {\rm d}s\\
=&0
\end{align*}
 that $E(t)=E(0)$ is a constant for $t\in [0, \infty)$.

Moreover, let  $\Gamma\subset\partial\Omega$  be an open set given by
\begin{align}\label{gamma}
\Gamma:=\big\{\boldsymbol x\in \partial\Omega:(\nabla l\cdot \boldsymbol \nu)>0\big\}\subset\partial\Omega.
\end{align}	

Thanks to Lemma \ref{le:car1} and the energy function defined in \eqref{energy1}, we derive the following observability inequality.

\begin{theo}\label{obineq}
Let $\boldsymbol u\in \boldsymbol{\mathcal{C}}^2([0, T]\times\Omega)$ be a solution of system \eqref{dual} and set $\boldsymbol v=e^{\tau l}\boldsymbol u$ with $\tau>0$ and
\begin{align*}
l(\boldsymbol x):=|\boldsymbol x-\boldsymbol x_0|^2/2\quad\text{for fixed } \boldsymbol x_0=(x^1_{0}, x^2_{0},\cdots, x^d_{0})^{\top}\in \mathbb{R}^d\backslash \overline{\Omega}.
\end{align*}
Suppose that
\[\mathscr C:=d-1,~\quad T>{2C_2C_3}/{C_0}.\]
If $\mu,\lambda$ belong to $\mathscr{H}$ and  $\rho$ satisfies that for some constant $\rho_2>0$,
\[1+(\nabla\rho\cdot\nabla l)/\rho>\rho_2>0,\]
then
\begin{align*}
E(0)\leq&\frac{C_1C_3}{TC_0-2C_2C_3}\int_0^T\int_{\Gamma}|\mu\partial_{\boldsymbol \nu}\boldsymbol u+\mu\nabla\boldsymbol u^\top\cdot\boldsymbol \nu+\lambda(\nabla\cdot\boldsymbol u)\boldsymbol \nu|^2 {\rm d}S {\rm d}t
\end{align*}
\text{as} $\tau\rightarrow0$, where
\begin{align*}
C_0&:=\min\{\rho_2, c_1\},\\
C_1&:=\max\big\{\sqrt 2, \sqrt{1+\mu_0}/(2\mu_0), \sqrt{5/(4\mu_0)}\big\}\max\{|\nabla l|\},\\
C_2&:=\max\big\{\rho_1/2, 2c_0^{-1}\max\{|\nabla l|^2\}\big\},\\
 C_3&:=\max\{1, 2\mu_1+\max\{|\lambda_0|, |\lambda_1|\}\}.
\end{align*}
\end{theo}

\begin{proof}
Taking integral of \eqref{a} over $(0, T)\times\Omega$  with $\boldsymbol w=\boldsymbol u$, we get
\begin{align}\label{carlemanestimate}
&\int_0^T\int_\Omega S_1\cdot S_2 {\rm d}\boldsymbol x  {\rm d}t\nonumber\\
\leq &\int_0^T\int_\Omega\frac{e^{2\tau l}}{2}\big|\rho\partial_{t}^2\boldsymbol u-\nabla\cdot(\mu(\nabla\boldsymbol u+\nabla\boldsymbol u^\top))-\nabla(\lambda\nabla\cdot\boldsymbol u)\big|^2 {\rm d}\boldsymbol x {\rm d}t=0.
\end{align}
We rewrite the left-hand side of \eqref{carlemanestimate} as
\begin{align*}
\text{L.H.S.}=I_1+I_2+I_3.
\end{align*}
Integration by parts and $\partial_t \boldsymbol v|_{(0, T)\times \partial\Omega}=0$,  we obtain
\begin{align}\label{I_1}
I_1=&\tau\int_0^T\int_\Omega\rho\partial^2_t\boldsymbol v\cdot(2\nabla\boldsymbol v\cdot\nabla l+\mathscr C\boldsymbol v) {\rm d}\boldsymbol x {\rm d}t\nonumber\\
=&\tau\int_\Omega\rho\partial_t\boldsymbol v\cdot(2\nabla\boldsymbol v\cdot\nabla l+\mathscr C\boldsymbol v) {\rm d}\boldsymbol x\big|_0^T-\tau\int_0^T\int_\Omega\rho\partial_t\boldsymbol v\cdot\partial_t(2\nabla\boldsymbol v\cdot\nabla l+\mathscr C\boldsymbol v) {\rm d}\boldsymbol x {\rm d}t\nonumber\\
=&\tau\int_\Omega\rho\partial_t\boldsymbol v\cdot(2\nabla\boldsymbol v\cdot\nabla l+\mathscr C\boldsymbol v) {\rm d}\boldsymbol x\big|_0^T-\tau\int_0^T\int_\Omega\nabla\cdot(\rho|\partial_t\boldsymbol v|^2\nabla l) {\rm d}\boldsymbol x {\rm d}t\nonumber\\
&+\tau\int_0^T\int_\Omega(\nabla\rho\cdot\nabla l)|\partial_t\boldsymbol v|^2 {\rm d}\boldsymbol x {\rm d}t+\tau\int_0^T\int_\Omega\rho\Delta l|\partial_t\boldsymbol v|^2 {\rm d}\boldsymbol x {\rm d}t\nonumber\\
&-\tau\int_0^T\int_\Omega\mathscr C\rho|\partial_t \boldsymbol v|^2{\rm d}\boldsymbol x {\rm d}t\nonumber\\
=&\tau\int_\Omega\rho\partial_t\boldsymbol v\cdot(2\nabla\boldsymbol v\cdot\nabla l+\mathscr C\boldsymbol v) {\rm d}\boldsymbol x\big|_0^T\nonumber\\
&+\tau\int_0^T\int_\Omega(\rho\Delta l+\nabla\rho\cdot\nabla l-\mathscr C\rho)|\partial_t\boldsymbol v|^2 {\rm d}\boldsymbol x {\rm d}t,
\end{align}
and
\begin{align*}
I_2
=&-\tau\int_0^T\int_\Omega\big(\nabla\cdot(\mu (\nabla\boldsymbol v+\nabla \boldsymbol v^{\top}))+\nabla(\lambda  \nabla\cdot\boldsymbol v)\big)\cdot(2\nabla\boldsymbol v\cdot\nabla l+\mathscr C\boldsymbol v) {\rm d}\boldsymbol x {\rm d}t\nonumber\\
=&-\tau\int_0^T\int_{\partial\Omega}((\mu (\nabla\boldsymbol v+\nabla \boldsymbol v^{\top})+\lambda (\nabla\cdot\boldsymbol v )I_d)\cdot\boldsymbol \nu)\cdot(2\nabla\boldsymbol v\cdot\nabla l+\mathscr C\boldsymbol v) {\rm d}S {\rm d}t\nonumber\\
&+\tau\int_0^T\int_\Omega(\mu (\nabla\boldsymbol v+\nabla \boldsymbol v^{\top})+\lambda(\nabla\cdot\boldsymbol v) I_d):\nabla(2\nabla\boldsymbol v\cdot\nabla l+\mathscr C\boldsymbol v) {\rm d}\boldsymbol x {\rm d}t\nonumber\\
=&-2\tau\int_0^T\int_{\partial\Omega}((\mu (\nabla\boldsymbol v+\nabla \boldsymbol v^{\top})+\lambda(\nabla\cdot\boldsymbol v)I_d)\cdot\boldsymbol \nu)\cdot(\nabla\boldsymbol v\cdot\nabla l) {\rm d}S {\rm d}t\nonumber\\
&+2\tau\int_0^T\int_\Omega((\mu (\nabla\boldsymbol v+\nabla \boldsymbol v^{\top})+\lambda(\nabla\cdot\boldsymbol v)I_d)\cdot\nabla\boldsymbol v):\nabla\nabla l{\rm d}\boldsymbol x {\rm d}t\\
&+\tau\mathscr C\int_0^T\int_\Omega(\mu (\nabla\boldsymbol v+\nabla \boldsymbol v^{\top})+\lambda(\nabla\cdot\boldsymbol v)I_d):\nabla\boldsymbol v{\rm d}\boldsymbol x {\rm d}t+\tau M,
\end{align*}
where
\begin{align}\label{sec2}
M=&\int_0^T\int_\Omega\nabla\cdot\big(((\mu (\nabla\boldsymbol v+\nabla \boldsymbol v^{\top})+\lambda(\nabla\cdot\boldsymbol v)I_d):\nabla \boldsymbol v )\nabla l\big){\rm d}\boldsymbol x {\rm d}t \nonumber\\
&-\int_0^T\int_\Omega\Delta l((\mu (\nabla\boldsymbol v+\nabla \boldsymbol v^{\top})+\lambda(\nabla\cdot\boldsymbol v)I_d):\nabla\boldsymbol v){\rm d}\boldsymbol x {\rm d}t \nonumber\\
&-\int_0^T\int_\Omega\underset{i=1}{\overset{d}{\sum}}((\partial_{x_i}\mu(\nabla\boldsymbol v+\nabla\boldsymbol v^\top) +\partial_{x_i}\lambda (\nabla\cdot \boldsymbol v)I_d):\nabla \boldsymbol v)\partial_{x_i}l{\rm d}\boldsymbol x {\rm d}t \nonumber\\
=&\int_0^T\int_{\partial\Omega}(\nabla l\cdot\boldsymbol\nu)(\mu (\nabla\boldsymbol v+\nabla \boldsymbol v^{\top})+\lambda(\nabla\cdot\boldsymbol v)I_d):\nabla \boldsymbol v{\rm d}S{\rm d}t\nonumber\\
&-\int_0^T\int_\Omega\Delta l((\mu (\nabla\boldsymbol v+\nabla \boldsymbol v^{\top})+\lambda(\nabla\cdot\boldsymbol v)I_d):\nabla\boldsymbol v){\rm d}\boldsymbol x {\rm d}t\nonumber\\
&-\int_0^T\int_\Omega\underset{i=1}{\overset{d}{\sum}}((\partial_{x_i}\mu(\nabla\boldsymbol v+\nabla\boldsymbol v^\top) +\partial_{x_i}\lambda (\nabla\cdot \boldsymbol v)I_d):\nabla \boldsymbol v)\partial_{x_i}l{\rm d}\boldsymbol x {\rm d}t.
\end{align}
Substituting \eqref{sec2} into $I_2$, we obtain
\begin{align}\label{I_2}
I_2=&-2\tau\int_0^T\int_{\partial\Omega}((\mu (\nabla\boldsymbol v+\nabla \boldsymbol v^{\top})+\lambda(\nabla\cdot\boldsymbol v)I_d)\cdot\boldsymbol \nu)\cdot(\nabla\boldsymbol v\cdot\nabla l) {\rm d}S {\rm d}t\nonumber\\
&+2\tau\int_0^T\int_\Omega((\mu (\nabla\boldsymbol v+\nabla \boldsymbol v^{\top})+\lambda(\nabla\cdot\boldsymbol v)I_d)\cdot\nabla\boldsymbol v):\nabla\nabla l{\rm d}\boldsymbol x {\rm d}t\nonumber\\
&+\tau\mathscr C\int_0^T\int_\Omega(\mu (\nabla\boldsymbol v+\nabla \boldsymbol v^{\top})+\lambda(\nabla\cdot\boldsymbol v)I_d):\nabla\boldsymbol v{\rm d}\boldsymbol x {\rm d}t\nonumber\\
&+\tau\int_0^T\int_{\partial\Omega}(\nabla l\cdot\boldsymbol\nu)(\mu (\nabla\boldsymbol v+\nabla \boldsymbol v^{\top})+\lambda  (\nabla\cdot\boldsymbol v)I_d):\nabla \boldsymbol v{\rm d}S{\rm d}t\nonumber\\
&-\tau\int_0^T\int_\Omega\Delta l((\mu (\nabla\boldsymbol v+\nabla \boldsymbol v^{\top})+\lambda(\nabla\cdot\boldsymbol v)I_d):\nabla\boldsymbol v){\rm d}\boldsymbol x {\rm d}t\nonumber\\
&-\tau\int_0^T\int_\Omega\underset{i=1}{\overset{d}{\sum}}((\partial_{x_i}\mu(\nabla\boldsymbol v+\nabla\boldsymbol v^\top) +\partial_{x_i}\lambda (\nabla\cdot \boldsymbol v)I_d):\nabla \boldsymbol v)\partial_{x_i}l{\rm d}\boldsymbol x {\rm d}t.
\end{align}

Moreover, 
\begin{align}\label{I_3}
I_3
=\tau^2 \mathcal R_1+\tau^3\mathcal R_2,
\end{align}
where $\mathcal R_1$ and $\mathcal R_2$ are the second and third order terms of $\tau$, respectively.

Plugging \eqref{I_1}, \eqref{I_2}, \eqref{I_3} and $\boldsymbol v=e^{\tau l}\boldsymbol u$ into \eqref{carlemanestimate}, and using \eqref{nessary} , we have
\begin{align}\label{zongin}
&2\int_0^T\int_{\partial\Omega}e^{2\tau l}((\mu (\nabla\boldsymbol u+\nabla \boldsymbol u^{\top})+\lambda ( \nabla\cdot\boldsymbol u)I_d)\cdot\boldsymbol \nu)\cdot(\nabla\boldsymbol u\cdot\nabla l) {\rm d}S {\rm d}t\nonumber\\
&-\int_0^T\int_{\partial\Omega}e^{2\tau l}(\nabla l\cdot\boldsymbol\nu)(\mu (\nabla\boldsymbol u+\nabla \boldsymbol u^{\top})+\lambda(\nabla\cdot\boldsymbol u)I_d):\nabla \boldsymbol u{\rm d}S{\rm d}t\nonumber\\
&+\tau|\mathcal R_3|+\tau^2|\mathcal R_4|+\tau^3|\mathcal R_5|\nonumber\\
\geq&\int_\Omega e^{2\tau l}\rho\partial_t\boldsymbol u\cdot(2\nabla\boldsymbol u\cdot\nabla l+\mathscr C\boldsymbol u) {\rm d}\boldsymbol x\big|_0^T+\int_0^T\int_\Omega e^{2\tau l}(\rho+(\nabla \rho\cdot\nabla l))|\partial_t\boldsymbol u|^2 {\rm d}\boldsymbol x {\rm d}t\nonumber\\
&-d\int_0^T\int_\Omega e^{2\tau l}((\mu (\nabla\boldsymbol u+\nabla \boldsymbol u^{\top})+\lambda(\nabla\cdot\boldsymbol u)I_d):\nabla\boldsymbol u){\rm d}\boldsymbol x {\rm d}t\nonumber\\
&-\int_0^T\int_\Omega e^{2\tau l}\underset{i=1}{\overset{d}{\sum}}((\partial_{x_i}\mu(\nabla\boldsymbol u+\nabla\boldsymbol u^\top) +\partial_{x_i}\lambda(\nabla\cdot \boldsymbol u)I_d):\nabla \boldsymbol u)\partial_{x_i}l{\rm d}\boldsymbol x {\rm d}t\nonumber\\
&+2\int_0^T\int_\Omega e^{2\tau l}((\mu (\nabla\boldsymbol u+\nabla \boldsymbol u^{\top})+\lambda(\nabla\cdot\boldsymbol u) I_d):\nabla\boldsymbol u){\rm d}\boldsymbol x {\rm d}t\nonumber\\
&+(d-1)\int_0^T\int_\Omega e^{2\tau l}((\mu (\nabla\boldsymbol u+\nabla \boldsymbol u^{\top})+\lambda(\nabla\cdot\boldsymbol u)I_d):\nabla\boldsymbol u){\rm d}\boldsymbol x {\rm d}t\nonumber\\
=&\int_\Omega e^{2\tau l}\rho\partial_t\boldsymbol u\cdot(2\nabla\boldsymbol u\cdot\nabla l+\mathscr C\boldsymbol u) {\rm d}\boldsymbol x\big|_0^T+\int_0^T\int_\Omega e^{2\tau l}(\rho+(\nabla \rho\cdot\nabla l))|\partial_t\boldsymbol u|^2 {\rm d}\boldsymbol x {\rm d}t\nonumber\\
&+\int_0^T\int_\Omega e^{2\tau l}((\mu (\nabla\boldsymbol u+\nabla \boldsymbol u^{\top})+\lambda (\nabla\cdot\boldsymbol u) I_d):\nabla\boldsymbol u){\rm d}\boldsymbol x {\rm d}t\nonumber\\
&-\int_0^T\int_\Omega e^{2\tau l}\underset{i=1}{\overset{d}{\sum}}((\partial_{x_i}\mu(\nabla\boldsymbol u+\nabla\boldsymbol u^\top) +\partial_{x_i}\lambda(\nabla\cdot \boldsymbol u)I_d):\nabla \boldsymbol u)\partial_{x_i}l{\rm d}\boldsymbol x {\rm d}t,
\end{align}
where $\mathcal R_3, \mathcal R_4$ and $\mathcal R_5$ are the first, second and third order terms of $\tau$.
The right-hand side of  \eqref{zongin} is not less then
\begin{align}\label{zong'}	
\int_\Omega e^{2\tau l}&\rho\partial_t\boldsymbol u\cdot(2\nabla\boldsymbol u\cdot\nabla l+\mathscr C\boldsymbol u) {\rm d}\boldsymbol x\big|_0^T\nonumber\\
&+\min\{\rho_2, c_1\}\int_0^T\int_\Omega e^{2\tau l}(|\sqrt\rho\partial_t\boldsymbol u|^2+|\nabla \boldsymbol u|^2) {\rm d}\boldsymbol x {\rm d}t.
\end{align}

Notice that  $\boldsymbol v=\partial_t\boldsymbol v=0$ on $(0, T)\times\partial\Omega$. Since $\boldsymbol u$ vanishes there, one has
\begin{align*}
\nabla \boldsymbol v=&\tau\boldsymbol v\cdot \nabla l^\top+e^{\tau l}\nabla\boldsymbol u=e^{\tau l}\nabla\boldsymbol u=e^{\tau l}(\nabla\boldsymbol u\cdot\boldsymbol \nu)\cdot\boldsymbol \nu^\top.
\end{align*}
Moreover, it is clear that
\begin{align}\label{bouequ}
	&((\mu (\nabla\boldsymbol u+\nabla \boldsymbol u^{\top})+\lambda  (\nabla\cdot\boldsymbol u)I_d)\cdot\boldsymbol \nu)\cdot(\nabla\boldsymbol u\cdot\boldsymbol \nu) \nonumber\\
	&\hspace{1cm}=(\mu (\nabla\boldsymbol u+\nabla \boldsymbol u^{\top})+\lambda ( \nabla\cdot\boldsymbol u)I_d):\nabla \boldsymbol u\geq c_0|\nabla\boldsymbol u|^2> 0~\quad \text{on}~\partial\Omega.
\end{align}
Then it follows from \eqref{gamma} and \eqref{bouequ} that the first and second term of left-hand side of \eqref{zongin} are equivalent to
\begin{align}\label{left}
&2\int_0^T\int_{\partial\Omega}e^{2\tau l}((\mu (\nabla\boldsymbol u+\nabla \boldsymbol u^{\top})+\lambda  (\nabla\cdot\boldsymbol u)I_d)\cdot\boldsymbol \nu)\cdot(\nabla\boldsymbol u\cdot\nabla l) {\rm d}S {\rm d}t\nonumber\\
&-\int_0^T\int_{\partial\Omega}e^{2\tau l}(\nabla l\cdot\boldsymbol\nu)(\mu (\nabla\boldsymbol u+\nabla \boldsymbol u^{\top})+\lambda  (\nabla\cdot\boldsymbol u)I_d):\nabla \boldsymbol u{\rm d}S{\rm d}t\nonumber\\
=&\int_0^T\int_{\partial\Omega}e^{2\tau l}(\nabla l\cdot\boldsymbol\nu)(\mu (\nabla\boldsymbol u+\nabla \boldsymbol u^{\top})+\lambda(\nabla\cdot\boldsymbol u)I_d):\nabla \boldsymbol u{\rm d}S{\rm d}t\nonumber\\
=&\int_0^T\int_{\partial\Omega}(\nabla l\cdot\boldsymbol\nu)e^{2\tau l}((\mu (\nabla\boldsymbol u+\nabla \boldsymbol u^{\top})+\lambda (\nabla\cdot\boldsymbol u)I_d)\cdot\boldsymbol \nu)\cdot\partial_\nu\boldsymbol u {\rm d}S {\rm d}t\nonumber\\
\leq&\int_0^T\int_{\Gamma}(\nabla l\cdot\boldsymbol\nu)e^{2\tau l}(\mu |\partial_{\boldsymbol \nu}\boldsymbol u|^2+\mu(\nabla \boldsymbol u^{\top}\cdot\boldsymbol \nu)\cdot\partial_\nu\boldsymbol u+\lambda  (\nabla\cdot\boldsymbol u)\boldsymbol \nu\cdot\partial_\nu\boldsymbol u) {\rm d}S {\rm d}t\nonumber\\
&+\int_0^T\int_{\Gamma}(\nabla l\cdot\boldsymbol\nu)e^{2\tau l}\big(\frac{\mu}{4}|\nabla\boldsymbol u^\top\cdot\boldsymbol \nu|^2+\frac{\mu}{4}|\partial_\nu\boldsymbol u|^2+\lambda^2|(\nabla\cdot\boldsymbol u)\boldsymbol\nu|^2\nonumber\\
&\qquad\qquad+(\frac{1}{2}(\nabla\boldsymbol u^\top\cdot\boldsymbol\nu)+\lambda(\nabla\cdot\boldsymbol u)\boldsymbol\nu)^2\big){\rm d}S {\rm d}t\nonumber\\
\leq &\max\big\{\sqrt 2, \sqrt{1+\mu_0}/(2\mu_0),  \sqrt{5/(4\mu_0)}\big\}\nonumber\\
&\cdot\max\{|\nabla l|\}\int_0^T\int_{\Gamma}e^{2\tau l}(|\mu \partial_{\boldsymbol\nu}\boldsymbol u+\mu\nabla \boldsymbol u^{\top}\cdot\boldsymbol \nu+\lambda  (\nabla\cdot\boldsymbol u)\boldsymbol \nu|^2 {\rm d}S {\rm d}t.
\end{align}
Furthermore, observe that
\begin{align*}
&\int_0^T\int_\Omega |2\nabla\boldsymbol u\cdot\nabla l+(d-1)\boldsymbol u|^2-|2\nabla\boldsymbol u\cdot\nabla l|^2{\rm d}\boldsymbol x{\rm d}t\nonumber\\
=&\int_0^T\int_\Omega 4(d-1)(\nabla\boldsymbol u\cdot\nabla l)\cdot\boldsymbol u+(d-1)^2|\boldsymbol u|^2{\rm d}\boldsymbol x{\rm d}t\nonumber\\
=&\int_0^T\int_\Omega 2(d-1)(\nabla|\boldsymbol u|^2\cdot\nabla l)+(d-1)^2|\boldsymbol u|^2{\rm d}\boldsymbol x{\rm d}t\nonumber\\
=&\int_0^T\int_\Omega 2(d-1)\nabla\cdot(|\boldsymbol u|^2\cdot\nabla l)-2(d-1)\Delta l|\boldsymbol u|^2+(d-1)^2|\boldsymbol u|^2{\rm d}\boldsymbol x{\rm d}t\nonumber\\
=&\int_0^T\int_\Omega -2d(d-1)|\boldsymbol u|^2+(d-1)^2|\boldsymbol u|^2{\rm d}\boldsymbol x{\rm d}t\nonumber\\
=&\int_0^T\int_\Omega (1-d^2)|\boldsymbol u|^2{\rm d}\boldsymbol x{\rm d}t\leq 0,
\end{align*}
which  leads to
\begin{align}\label{r1}
&\big| \int_\Omega e^{2\tau l}\rho\partial_t\boldsymbol u\cdot(2\nabla\boldsymbol u\cdot\nabla l+\mathscr C\boldsymbol u\big) {\rm d}\boldsymbol x\big|\nonumber\\
= &\big|\int_\Omega e^{2\tau l}\rho\partial_t\boldsymbol u\cdot(2\nabla\boldsymbol u\cdot\nabla l+(d-1)\boldsymbol u\big) {\rm d}\boldsymbol x\big|\nonumber\\
 \leq & \int_\Omega e^{2\tau l}(|\rho\partial_t\boldsymbol u|^2/2+|\nabla \boldsymbol u\cdot\nabla l|^2){\rm d}\boldsymbol x\nonumber\\
 \leq & \int_\Omega e^{2\tau l}(\rho_1/2|\sqrt\rho\partial_t\boldsymbol u|^2+2c_0^{-1}c_0|\nabla l|^2|\nabla \boldsymbol u|^2){\rm d}\boldsymbol x\nonumber\\
 \leq &\max\big\{\rho_1/2, 2c_0^{-1}\max\{|\nabla l|^2\}\big\}\int_\Omega e^{2\tau l}(|\sqrt\rho\partial_t\boldsymbol u|^2+c_0|\nabla \boldsymbol u|^2){\rm d}\boldsymbol x\nonumber\\
 \leq & \max\big\{\rho_1/2, 2c_0^{-1}\max\{|\nabla l|^2\}\big\}e^{2\tau \max\{l\}}E(t).
\end{align}

In addition,  we have
\begin{align}\label{r2}
 e^{2\tau \min \{l\}}E(t)\leq &\int_\Omega e^{2\tau l}(|\sqrt \rho\partial_t\boldsymbol u|^2+(2\mu_1+\max\{|\lambda_0|, |\lambda_1|\})|\nabla\boldsymbol u|^2){\rm d}\boldsymbol x\nonumber\\
\leq &\max\big\{1, 2\mu_1+\max\{|\lambda_0|, |\lambda_1|\}\big\}\int_\Omega e^{2\tau l}(|\sqrt \rho\partial_t\boldsymbol u|^2+|\nabla\boldsymbol u|^2){\rm d}\boldsymbol x.
\end{align}
Substituting  \eqref{zong'}, \eqref{left}--\eqref{r2} into \eqref{zongin} yields that
\begin{align*}
	&\bigg(\frac{T\min\{\rho_2, c_1\}}{\max\{1, 2\mu_1+\max\{|\lambda_0|, |\lambda_1|\}\}}-2\max\big\{\rho_1/2, 2c_0^{-1}\max\{|\nabla l|^2\}\big\}\bigg)E(0)\\
	\leq &\max\big\{\sqrt 2, \sqrt{1+\mu_0}/(2\mu_0),  \sqrt{5/(4\mu_0)}\big\}\\
&\cdot\max\{|\nabla l|\}\int_0^T\int_{\Gamma}(|\mu \partial_{\boldsymbol\nu}\boldsymbol u+\mu\nabla \boldsymbol u^{\top}\cdot\boldsymbol \nu+\lambda  (\nabla\cdot\boldsymbol u)\boldsymbol \nu|^2 {\rm d}S {\rm d}t
\end{align*}
as $\tau$ tends to 0.

The proof is complete.
\end{proof}

Based on Theorem \ref{obineq}, we further give the following fact.
\begin{theo}
Suppose that $l, \Gamma, T, \rho, \mu,\lambda$ satisfy the conditions of Theorem \ref{obineq}.
Let $U$ be a bounded $\mathcal C^2$ neighborhood of $\rho$.
Then there is a $\mathcal{C}^1$ neighborhood $V$ of $\rho$ and some constant $C>0$ satisfying the following: for all $\tilde\rho\in U$ such that $\tilde\rho\in V$, the solutions
\begin{align*}
\tilde{\boldsymbol u}\in \mathcal{C}([0, T]; \boldsymbol H^{1}(\Omega))\cap\mathcal{C}^1([0, T];\boldsymbol L^{2}(\Omega))
\end{align*}
of the problem
\begin{align*}
\begin{cases}
\partial_t^2\tilde{\boldsymbol u}(t,\boldsymbol x)-\mathcal L_{\tilde{\rho}}\tilde{\boldsymbol u}(t,\boldsymbol x)=\boldsymbol 0 \quad &\text{in}~(0, T)\times\Omega,\\
\tilde{\boldsymbol u}(t,\boldsymbol x)=\boldsymbol 0,\quad &\text{on}~(0, T)\times\partial\Omega
\end{cases}
\end{align*}
satisfies the observability inequality
\begin{align*}
\|\partial_t\tilde{\boldsymbol u}(0)&\|_{\boldsymbol L^2(\Omega)}+\|\tilde{\boldsymbol u}(0)\|_{\boldsymbol H_0^1(\Omega)}
\leq C\|\mu\partial_{\boldsymbol \nu}\tilde{\boldsymbol u}+\mu\nabla\tilde{\boldsymbol u}^\top\cdot\boldsymbol \nu+\lambda(\nabla\cdot\tilde{\boldsymbol u})\boldsymbol \nu\|_{\boldsymbol L^2((0, T)\times\Gamma)}.
\end{align*}
\end{theo}

\begin{rema}
We only need to prove that the assumptions $1+(\nabla \tilde\rho\cdot\nabla l)/\tilde\rho>\rho_2$ , the boundary restriction \eqref{gamma} and the constants $C_i, i=0, 1, 2, 3$ also still hold for $\tilde\rho$.
In fact, if we just choose $V$ smaller enough, then the result is obvious.
\end{rema}
\section{Lipschitz stability for the medium}\label{5}
In this section, we will use the reconstruction formula to prove the locally Lipschitz stable estimate and the logarithmic estimate for low and high frequencies, respectively.
\subsection{A estimate for the control operator}
The following lemma gives the estimate for the control operator.
\begin{lemm}
Let $\mathcal W_{\rho, T}: \boldsymbol L^2((0, T)\times \partial\Omega)\rightarrow \boldsymbol  L^2(\Omega)$ be the control map given by \eqref{ctm}.
Then its adjoint operator
\begin{align*}
\mathcal W_{\rho, T}^* \boldsymbol \phi&=\mu\partial_{\boldsymbol \nu}\boldsymbol p+\mu\nabla \boldsymbol p^\top\cdot\boldsymbol \nu+\lambda(\nabla\cdot\boldsymbol p)\boldsymbol \nu\big|_{((0, T)\times\partial\Omega)},
\end{align*}
where $\boldsymbol p$ is the solution of
\begin{align}\label{dual sys}
\begin{cases}
\partial_t^2\boldsymbol p(t,\boldsymbol x)-\mathcal L_\rho\boldsymbol p(t,\boldsymbol x)=\boldsymbol 0 \quad &\text{in}~ (0, T)\times\Omega,\\
\boldsymbol p(T, \boldsymbol x)=\boldsymbol 0, \quad \partial_t\boldsymbol p (T, \boldsymbol x)=\boldsymbol \phi ~\quad &\text{in}~  \Omega,\\
\boldsymbol p(t, \boldsymbol x)=\boldsymbol 0  \quad &\text{on}~  (0, T)\times\partial\Omega,
\end{cases}
\end{align}
satisfies
\begin{align}\label{conj}
\|(\mathcal W_{\rho, T}^*)^{-1}\|_{R(\mathcal W_{\rho, T}^*)\rightarrow \boldsymbol L^2(\Omega)}\leq C_{obs}.
\end{align}
\end{lemm}
\begin{proof}
From the initial-boundary value problem \eqref{ab} and \eqref{dual sys}, we calculate that
\begin{align*}
0=&\int_0^T(\partial_t^2\boldsymbol u_f-\mathcal L_\rho\boldsymbol u_f,  \boldsymbol p)_{\boldsymbol L^2(\Omega;  \rho {\rm d}\boldsymbol x)}-(\boldsymbol u_f,  \partial_t^2\boldsymbol p-\mathcal L_\rho \boldsymbol p)_{\boldsymbol L^2(\Omega;  \rho {\rm d}\boldsymbol x)} {\rm d}t\\
=&\int_\Omega\partial_t\boldsymbol u_f\cdot\boldsymbol p-\boldsymbol u_f\cdot\partial_t\boldsymbol p\big|^T_0\rho  {\rm d}\boldsymbol x\\
&-\int^T_0\int_{\partial\Omega}\mu(\partial_{\boldsymbol \nu }\boldsymbol u_f+\nabla\boldsymbol u_f^\top\cdot\boldsymbol \nu)\cdot \boldsymbol p-\boldsymbol u_f\cdot\mu (\partial_{\boldsymbol \nu }\boldsymbol p+\nabla\boldsymbol p^\top\cdot\boldsymbol \nu) {\rm d}S {\rm d}t\\
&-\int^T_0\int_{\partial\Omega}\lambda(\nabla\cdot \boldsymbol u_f)\boldsymbol \nu\cdot\boldsymbol p-\boldsymbol u_f\cdot\lambda(\nabla\cdot \boldsymbol p)\boldsymbol \nu {\rm d}S {\rm d}t\\
=&-(\boldsymbol u_f(T), \partial_t\boldsymbol p(T))_{\boldsymbol L^2(\Omega;  \rho {\rm d}\boldsymbol x)}+\langle\boldsymbol f,  \mu\partial_{\boldsymbol \nu} \boldsymbol p+\mu\nabla\boldsymbol p^\top\cdot\boldsymbol \nu+\lambda(\nabla\cdot \boldsymbol p)\boldsymbol \nu\rangle_{\boldsymbol L^2((0, T)\times\partial\Omega)}.
\end{align*}
Vanishing initial-boundary conditions yield that
\begin{align*}
\langle\boldsymbol f, \mu\partial_{\boldsymbol \nu} \boldsymbol p+\mu\nabla\boldsymbol p^\top\cdot\boldsymbol \nu+\lambda(\nabla\cdot \boldsymbol p)\boldsymbol \nu\rangle_{\boldsymbol L^2((0, T)\times\partial\Omega)}&=(\boldsymbol u_f(T),  \boldsymbol \phi)_{\boldsymbol L^2(\Omega;  \rho{\rm d}\boldsymbol x)}\\
&=(\mathcal W_{\rho, T}\boldsymbol f,  \boldsymbol \phi)_{\boldsymbol L^2(\Omega;  \rho {\rm d}\boldsymbol x)}\\
&=\langle\boldsymbol f,  \mathcal W_{\rho, T}^*\boldsymbol \phi\rangle_{\boldsymbol L^2((0, T)\times\partial\Omega)}.
\end{align*}
That is,
\begin{align*}
	\boldsymbol \phi=(\mathcal W_{\rho, T}^*)^{-1}(\mu\partial_{\boldsymbol \nu} \boldsymbol p+\mu\nabla\boldsymbol p^\top\cdot\boldsymbol \nu+\lambda(\nabla\cdot \boldsymbol p)\boldsymbol \nu).
\end{align*}
It follows from \eqref{obs} that
\begin{align*}
\|\boldsymbol \phi\|_{\boldsymbol L^2(\Omega)}\leq&C_{obs}\|\mu\partial_{\boldsymbol \nu}\boldsymbol p+\mu\nabla\boldsymbol p^\top\cdot\boldsymbol \nu+\lambda(\nabla\cdot\boldsymbol p)\boldsymbol \nu\|_{\boldsymbol L^2((0, T)\times\partial\Omega)}\\
=&C_{obs}\|\mathcal W_{\rho, T}^*\boldsymbol \phi\|_{\boldsymbol L^2((0, T)\times\partial\Omega)}.
\end{align*}
Observe that $\mathcal W_{\rho, T}$ is surjective, which leads to $\ker(\mathcal W_{\rho, T}^*)=R(\mathcal W_{\rho, T})^{\bot}=\{0\}$. Thus  $\mathcal W_{\rho, T}^*$ is injective
 and $(\mathcal W_{\rho, T}^*)^{-1}:R(\mathcal W_{\rho, T}^*)\rightarrow \boldsymbol L^2(\Omega)$ with
\begin{align*}
\|(\mathcal W_{\rho, T}^*)^{-1}\|_{R(\mathcal W_{\rho, T}^*)\rightarrow \boldsymbol L^2(\Omega)}\leq C_{obs}.
\end{align*}
This ends the proof of the lemma.
\end{proof}

By using \eqref{conj} and
\begin{align*}
(\mathcal W_{\rho, T}^\dag)^{*}=(\mathcal W_{\rho, T}^*)^{\dag}=(\mathcal W_{\rho, T}^*)^{-1}P_{R(\mathcal W_{\rho, T}^*)},
\end{align*}
we obtain
\begin{align*}
 \|\mathcal W_{\rho, T}^\dag\|_{\boldsymbol L^2(\Omega)\rightarrow \boldsymbol L^2((0, T)\times\partial\Omega)}=\|(\mathcal W_{\rho, T}^*)^\dag\|_{\boldsymbol L^2((0, T)\times\partial\Omega)\rightarrow \boldsymbol L^2(\Omega)}\leq C_{obs},
\end{align*}
where $P_{R(\mathcal W_{\rho, T}^*)}$ is the orthogonal projection onto $R(\mathcal W_{\rho, T}^*)$.

Consequently, given by  $(\mathcal \mathcal W_{\rho, T}^*\mathcal W_{\rho, T})^\dag=\mathcal W_{\rho, T}^\dag(\mathcal W_{\rho, T}^*)^\dag$, it holds that
\begin{align*}
 \|\mathcal J^\dag(\Lambda_{\rho, 2T})\|_{\boldsymbol L^2((0, T)\times\partial\Omega)\rightarrow \boldsymbol L^2((0, T)\times\partial\Omega)}=&\|(\mathcal W_{\rho, T}^*\mathcal W_{\rho, T})^\dag\|_{\boldsymbol L^2((0, T)\times\partial\Omega)\rightarrow \boldsymbol L^2((0, T)\times\partial\Omega)}\\
 \leq& C^2_{obs}.
\end{align*}

\begin{lemm}
 Two operators $\mathcal J(\Lambda_{\rho, 2T})$ and $\tilde{\mathcal J}(\tilde{\Lambda}_{\tilde{\rho}, 2T})$ satisfy
 \begin{align}\label{im}
  &\|\tilde{\mathcal J}^\dag(\tilde{\Lambda}_{\tilde{\rho}, 2T})-\mathcal J^\dag(\Lambda_{\rho, 2T})\|_{\boldsymbol L^2((0, T)\times\partial\Omega)\rightarrow \boldsymbol L^2((0, T)\times\partial\Omega)}\nonumber\\
  &\hspace{0.5cm}\leq 3C_{obs}^4\|\tilde{\mathcal J}(\tilde {\Lambda}_{\tilde{\rho}, 2T})-\mathcal J(\Lambda_{\rho, 2T})\|_{\boldsymbol L^2((0, T)\times\partial\Omega)\rightarrow \boldsymbol L^2((0, T)\times\partial\Omega)}.
   \end{align}
\end{lemm}
\begin{proof}
 The process of proof can be obtained by referring to Theorem 3.2--3.3 in \cite{Stewart1977} and Lemma 2.1 in \cite{Izumino1983}.
\end{proof}

Under above estimates and the assumption of stable observability for elastic wave equation, we derive the following Lipschitz estimate for low frequencies.
\begin{theo}\label{t2}
Let $R>0$ with $\Omega\subset B(0, R)$ and set
  \begin{align*}
\mathcal E:=&\|\tilde{\mathcal J}(\tilde {\Lambda}_{\tilde{\rho}, 2T})-\mathcal J(\Lambda_{\rho, 2T})\|_{\boldsymbol L^2((0, T)\times\partial\Omega)\rightarrow \boldsymbol L^2((0, T)\times\partial\Omega)}\\
 	&+\|\tilde{\Lambda}_{\tilde{\rho}, T}-\Lambda_{\rho, T}\|_{\boldsymbol H^1_{0}((0, T)\times\partial\Omega)\rightarrow \boldsymbol L^2((0, T)\times\partial\Omega)}.
 \end{align*}
 Suppose that the system \eqref{ab} is stably observable for the density $\rho$ in a set $U\subset \mathcal C^\infty(\Omega)$ from $\partial \Omega$ in time $T>0$.
Then for all $\tilde{\rho}\in U$, we have
 \begin{align*}
  \big|\mathcal{F}(\tilde{\rho}-\rho)(\boldsymbol \xi)\big|\leq& Ce^{2R|\boldsymbol \xi|}\mathcal E,
 \end{align*}
 where $\mathcal F(\cdot)(\boldsymbol \xi)$ denotes the Fourier transform of the extension by zero of $\cdot$ onto $\mathbb R^d$ and $C>0$ is a constant depending on $\Omega, T,  \rho$ and $C_{obs}$.
\end{theo}

\begin{proof}
We will omit writing $\boldsymbol L^2((0, T)\times\partial\Omega)$ as a subscript.
 Let $\rho, \tilde{\rho}\in U$. Then using Theorem \ref{t1} and the identity \eqref{af}  yields that
 \begin{align}\label{ineq9}
 &|\boldsymbol \iota|^2\big|\mathcal{F}(\tilde{\rho})-\mathcal{F}(\rho)\big|\\
 =&\big|(\boldsymbol \phi,  \boldsymbol \psi)_{\boldsymbol L^2(\Omega; \tilde{\rho}{\rm d}\boldsymbol x)}-(\boldsymbol \phi,  \boldsymbol \psi)_{\boldsymbol L^2(\Omega; \rho {\rm d}\boldsymbol x)}\big|\nonumber\\
 =&\big|\langle\tilde{\mathcal J}^\dag(\tilde {\Lambda}_{\tilde{\rho}, 2T}) \tilde{\mathcal K}(\tilde {\Lambda}_{\tilde{\rho}, T}) \boldsymbol \phi,  \tilde{\mathcal K}(\tilde {\Lambda}_{\tilde{\rho}, T}) \boldsymbol \psi\rangle-\langle\mathcal  J^\dag(\Lambda_{\rho, 2T}) \mathcal K(\Lambda_{\rho, T}) \boldsymbol \phi,  \mathcal K(\Lambda_{\rho, T}) \boldsymbol \psi\rangle\big|\nonumber\\
  \leq&\big|\langle(\tilde{\mathcal J}^\dag(\tilde {\Lambda}_{\tilde{\rho}, 2T})-\mathcal J^\dag(\Lambda_{\rho, 2T}))\mathcal K(\Lambda_{\rho, T})\boldsymbol \phi,  \mathcal K(\Lambda_{\rho, T})\boldsymbol \psi\rangle\big|\nonumber\\
  &+\big|\langle\tilde{\mathcal J}^\dag(\tilde {\Lambda}_{\tilde{\rho}, 2T}) \mathcal K(\Lambda_{\rho, T}) \boldsymbol \phi,  (\tilde{\mathcal K}(\tilde{\Lambda}_{\tilde{\rho}, T})-\mathcal K(\Lambda_{\rho, T}))\boldsymbol \psi\rangle\big|\nonumber\\
  &+\big|\langle\tilde{\mathcal J}^\dag(\tilde {\Lambda}_{\tilde{\rho}, 2T})(\tilde{\mathcal K}(\tilde{\Lambda}_{\tilde{\rho}, T})-\mathcal K(\Lambda_{\rho, T}))\boldsymbol \phi,  \tilde{\mathcal K}(\tilde{\Lambda}_{\tilde{\rho}, T})\boldsymbol \psi\rangle\big|.
 \end{align}
Since
 \begin{align*}
  (\mathscr IT_0\boldsymbol \phi)(t, \boldsymbol x)=(T-t)\boldsymbol \phi(\boldsymbol x),\quad t\in [0, T],~\boldsymbol x\in \partial\Omega,
 \end{align*}
there is $C_{T, \Omega}>0$ depending only on $T$ and $\Omega$ such that
 \begin{align*}
  \|(\tilde{\mathcal K}(\tilde{\Lambda}_{\tilde{\rho}, T})-\mathcal K(\Lambda_{\rho, T}))\boldsymbol \phi\|=&\|(\tilde\Lambda_{\tilde \rho, T}^*-\Lambda_{\rho, T}^*)\mathscr IT_0\boldsymbol \phi\|\\
  \leq&C_{T, \Omega}\|\tilde\Lambda_{\tilde \rho, T}^*-\Lambda_{\rho, T}^*\|_{\boldsymbol H^1_0((0, T)\times\partial\Omega)\rightarrow \boldsymbol L^2((0, T)\times\partial\Omega)}\|\boldsymbol \phi\|_{\boldsymbol{\mathcal C}^1(\partial\Omega)}\\
  =&C_{T, \Omega}\|\tilde \Lambda_{\tilde \rho, T}-\Lambda_{\rho, T}\|_{\boldsymbol H^1_0((0, T)\times\partial\Omega)\rightarrow \boldsymbol L^2((0, T)\times\partial\Omega)}\|\boldsymbol \phi\|_{\boldsymbol {\mathcal C}^1(\partial\Omega)}.
 \end{align*}

Let $C_{\rho}:=\|\mathcal W_{\rho, T}^*\|_{\boldsymbol L^2(\Omega)\rightarrow \boldsymbol L^2((0, T)\times\partial\Omega)}$\textcolor{blue}. Then
 \begin{align}\label{auxi1}
  \|\mathcal K(\Lambda_{\rho, T}) \boldsymbol \phi\|=\|\mathcal W_{\rho, T}^*\boldsymbol \phi\|\leq C_{\rho}\|\boldsymbol \phi\|_{\boldsymbol L^2(\Omega;  \rho {\rm d} \boldsymbol x)}.
 \end{align}
 It follows from \eqref{im} and \eqref{auxi1} that
 \begin{align}\label{a1}
 & \big|\langle(\tilde{\mathcal J}^\dag(\tilde{\Lambda}_{\tilde\rho, 2T})-\mathcal J^\dag(\Lambda_{\rho, 2T}))\mathcal K(\Lambda_{\rho, T}) \boldsymbol \phi, \mathcal K(\Lambda_{\rho, T})  \boldsymbol \psi\rangle\big|\nonumber\\
  &\hspace{0.5cm}\leq3C_{obs}^4
  C_{\rho}^2\|\tilde{\mathcal J}(\tilde{\Lambda}_{\tilde\rho, 2T})-\mathcal J(\Lambda_{\rho, 2T})\|\|\boldsymbol \phi\|_{\boldsymbol L^2(\Omega; \rho {\rm d}\boldsymbol x)}\|\boldsymbol \psi\|_{\boldsymbol L^2(\Omega; \rho{\rm d} \boldsymbol x)},
  \end{align}
  and
  \begin{align}\label{a2}
  &\big|\langle \tilde{\mathcal J}^\dag(\tilde{\Lambda}_{\tilde\rho, 2T})\mathcal K(\Lambda_{\rho, T}) \boldsymbol \phi, (\tilde{\mathcal K}(\tilde{\Lambda}_{\tilde\rho, T}) -\mathcal K(\Lambda_{\rho, T}) )\boldsymbol \psi\rangle\big|\nonumber\\
  &\hspace{0.5cm}\leq C_{obs}^2C_{\rho}C_{T, \Omega}\|\tilde\Lambda_{\tilde\rho, T}-\Lambda_{\rho, T}\|_{\boldsymbol H^1_0((0, T)\times\partial\Omega)\rightarrow \boldsymbol  L^2((0, T)\times\partial\Omega)}\|\boldsymbol \phi\|_{\boldsymbol L^2(\Omega; \rho {\rm d} \boldsymbol x)}\|\boldsymbol \psi\|_{\boldsymbol {\mathcal C}^1(\partial\Omega)}.
 \end{align}
Observe that $(\tilde{\mathcal J}^\dag(\tilde{\Lambda}_{\tilde\rho, 2T}))^*=\tilde{\mathcal J}^\dag(\tilde{\Lambda}_{\tilde\rho, 2T})$. Then
 \begin{align}\label{a3}
  &\big|\langle \tilde{\mathcal J}^\dag(\tilde{\Lambda}_{\tilde\rho, 2T})(\tilde{\mathcal K}(\tilde{\Lambda}_{\tilde\rho, T})-\mathcal K(\Lambda_{\rho, T}))\boldsymbol \phi, \tilde{\mathcal K} (\tilde{\Lambda}_{\tilde\rho, T})\boldsymbol \psi\rangle
  \big|\nonumber\\
  =&\big|\langle(\tilde{\mathcal K}(\tilde{\Lambda}_{\tilde\rho, T})-\mathcal K(\Lambda_{\rho, T}))\boldsymbol \phi, \tilde{\mathcal W}_{\tilde\rho, T}^\dag \boldsymbol \psi\rangle\big|\nonumber\\
  \leq& C_{T, \Omega}C_{obs}\|\tilde\Lambda_{\tilde \rho, T}-\Lambda_{\rho, T}\|_{\boldsymbol H^1_0((0, T)\times\partial\Omega)\rightarrow \boldsymbol L^2((0, T)\times\partial\Omega)}\|\boldsymbol \phi\|_{\boldsymbol{\mathcal{C}}^1(\partial\Omega)}\|\boldsymbol \psi\|_{\boldsymbol L^2(\Omega; \tilde{\rho}{\rm d} \boldsymbol x)}.
 \end{align}
Substituting \eqref{a1}--\eqref{a3} into \eqref{ineq9} shows that
 \begin{align*}
  &\big|(\boldsymbol \phi, \boldsymbol \psi)_{\boldsymbol L^2(\Omega; \tilde{\rho} {\rm d}\boldsymbol x)}-(\boldsymbol \phi, \boldsymbol \psi)_{\boldsymbol L^2(\Omega; \rho{\rm d} \boldsymbol x)}\big|\nonumber\\
  \leq& C(\|\tilde{\mathcal J}(\tilde{\Lambda}_{\tilde\rho, 2T})-\mathcal J(\Lambda_{\rho, 2T})\|_{\boldsymbol L^2((0, T)\times\partial\Omega)\rightarrow \boldsymbol L^2((0, T)\times\partial\Omega)}\nonumber\\
  &+\|\tilde\Lambda_{\tilde \rho, T}-\Lambda_{\rho, T}\|_{\boldsymbol H^1_0((0, T)\times\partial\Omega)\rightarrow \boldsymbol L^2((0, T)\times\partial\Omega)})\|\boldsymbol \phi\|_{\boldsymbol{\mathcal C}^1(\Omega)}\|\boldsymbol \psi\|_{\boldsymbol{\mathcal C}^1(\Omega)},
 \end{align*}
where the constant $C>0$ depends on $\Omega, T, \rho$ and $C_{obs}$.

Taking the derivative  of $\boldsymbol \phi:=\boldsymbol\iota  e^{{\rm i}\boldsymbol \theta\cdot \boldsymbol x}$ with respect to $\boldsymbol x$,
 we have
 \[ |\partial_{x_i}\boldsymbol \phi(\boldsymbol x)|\leq 2|\boldsymbol\iota||\boldsymbol\xi|e^{R|\boldsymbol\xi|/2}\leq C_R |\boldsymbol\iota| e^{R|\boldsymbol\xi|},\]
 where $R$ is the radius of a ball $B(0, R)$ that contains $\Omega$, and $C_R>0$ is a constant depending only on $R$.
 Thus,
 \begin{align*}
  |\boldsymbol \iota|^2\big|\mathcal{F}(\tilde{\rho}-\rho)(\boldsymbol\xi)\big|\leq C|\boldsymbol\iota|^2e^{2R|\boldsymbol\xi|}(&\|\tilde{\mathcal J}(\tilde{\Lambda}_{\tilde\rho, 2T})-\mathcal J(\Lambda_{\rho, 2T})\|_{\boldsymbol L^2((0, T)\times\partial\Omega)\rightarrow \boldsymbol L^2((0, T)\times\partial\Omega)}\\
  &+\|\tilde{\Lambda}_{\tilde \rho, T}-\Lambda_{\rho, T}\|_{\boldsymbol H^1_{0}((0, T)\times\partial\Omega)\rightarrow \boldsymbol L^2((0, T)\times\partial\Omega)}).
 \end{align*}
The proof is complete.
\end{proof}

We now give a logarithmic stability result for high frequencies.
Before that, we need an additional assumption on $\rho$. We take the extension by zero of $\rho$ onto $\mathbb R^d$ and still remark as $\rho$. Define the space as follows
\begin{align*}
 \mathscr Q:&=\left\{\rho\in H^{s+2}(\mathbb R^d), ~\|\rho\|_{H^{s+2}(\mathbb R^d)}\leq C,  ~\text{for some }s\geq0 ~\text{and}~ C>0\right\}.
\end{align*}
\begin{coro}
 Under the same assumptions of  Theorem \ref{t2}.
 If $\mathcal E$ is sufficiently small and $ \rho, \tilde\rho\in \mathscr Q$, then
  \begin{align*}
  \|\tilde{\rho}-\rho\|_{\boldsymbol L^2(\Omega)}\leq C\big(-\ln \mathcal E\big)^{-2}.
 \end{align*}
\end{coro}
\begin{proof}
One has that for $\gamma> 2$,
\begin{align}\label{ineq11}
	\|(\tilde\rho-\rho)(\boldsymbol x)\|^2_{\boldsymbol L^2(\Omega)}\leq\|\mathcal F(\tilde\rho-\rho)(\boldsymbol \xi)\|^2_{\boldsymbol L^2(\mathbb R^d)}
	=P_1+P_2.
\end{align}
where
\begin{align*}
P_1:=\int_{|\boldsymbol\xi|\leq \gamma}|\mathcal F(\tilde\rho-\rho)(\boldsymbol \xi)|^2 {\rm d}\boldsymbol \xi,\quad P_2:=\int_{|\boldsymbol\xi|>\gamma}|\mathcal F(\tilde\rho-\rho)(\boldsymbol \xi)|^2 {\rm d}\boldsymbol \xi.
\end{align*}
Note that $\gamma^d\leq d!e^\gamma$ and $\rho, \tilde \rho\in \mathscr Q$. By a direct calculation, we have
\begin{align}\label{P1}
P_1\leq& C\int_{|\boldsymbol \xi|\leq \gamma}e^{4R|\boldsymbol\xi|}\mathcal E^2 {\rm d}\boldsymbol \xi\leq C\gamma^de^{4R\gamma}\mathcal E^2
\leq Ce^{4R\gamma+\gamma}\mathcal E^2,
\end{align}
and
\begin{align}\label{P2}
	P_2=&\int_{|\boldsymbol\xi|>\gamma}\frac{1}{(1+|\boldsymbol \xi|^2)^{s+2}}(1+|\boldsymbol \xi|^2)^{s+2}|\mathcal{F}(\tilde{\rho}-\rho)(\boldsymbol\xi)|^2 {\rm d}\boldsymbol \xi
	\leq \frac{C}{1+\gamma^{2s+4}}\leq\frac{C}{\gamma^{4}}.
\end{align}
Substituting \eqref{P1} and \eqref{P2} into \eqref{ineq11} yields that
\begin{align}\label{ineq11'}
	\|(\tilde\rho-\rho)(\boldsymbol x)\|^2_{\boldsymbol L^2(\Omega)}
\leq&C\big(\frac{1}{\gamma^4}+e^{4R\gamma+\gamma}\mathcal E^2\big).
\end{align}

Let $\delta>0$ with
	$\mathcal E< \delta\leq e^{-8R-2}$.
We take
\begin{align*}
\gamma: =-\frac{1}{4R+1}\ln\mathcal E.
\end{align*}
It is clear that $\gamma>2$. By substituting the expression of $\gamma$ into \eqref{ineq11'}, we obtain
\begin{align*}
	\|(\tilde\rho-\rho)(\boldsymbol x)\|^2_{\boldsymbol L^2(\Omega)}\leq C\big((-\ln\mathcal E)^{-4}+\mathcal E\big)\leq C\big(-\ln\mathcal E\big)^{-4}.
\end{align*}
This proof is complete.
\end{proof}

\section*{Acknowledgment}
The second    author expresses deep gratitude to Prof. Peijun Li  for very  valuable  discussions.



\end{document}